\documentclass[12pt]{article}

\usepackage[usenames]{color}
\usepackage{amsthm,amssymb,amsmath,amscd}
\usepackage{fullpage}

\usepackage[shortlabels]{enumitem}

\usepackage[colorlinks=true,
linkcolor=webgreen,
filecolor=webbrown,
citecolor=webgreen]{hyperref}

\definecolor{webgreen}{rgb}{0,.5,0}
\definecolor{webbrown}{rgb}{.6,0,0}

\theoremstyle{plain}
\newtheorem{theorem}{Theorem}
\newtheorem{corollary}[theorem]{Corollary}
\newtheorem{lemma}[theorem]{Lemma}
\newtheorem{proposition}[theorem]{Proposition}

\theoremstyle{definition}

\newtheorem{question}[theorem]{Question}
\newtheorem{notation}[theorem]{Notation}

\theoremstyle{remark}
\newtheorem{remark}[theorem]{Remark}

\def\modd#1 #2{#1\ \mbox{\rm (mod}\ #2\mbox{\rm )}}
\def\Zee{\mathbb{Z}}
\def\Enn{\mathbb{N}}

\newcommand{\seqnum}[1]{\href{https://oeis.org/#1}{\rm \underline{#1}}}

\usepackage{color, graphicx, comment, amsmath}
\usepackage{tikz}

\allowdisplaybreaks

\font\KFracFont=cmr12 at 20pt
\def\KKK{\mathop{\lower 4pt\hbox{\KFracFont K}}\limits}

\usepackage{tikz}
\usepackage{epsf}
\usepackage[english]{babel}

\DeclareMathOperator{\sign}{sgn}

\author{J.-P. Allouche \\
CNRS, IMJ-PRG \\
Sorbonne Universit\'e \\
4 Place Jussieu \\
F-75252 Paris Cedex 05 France \\
\href{mailto:jean-paul.allouche@imj-prg.fr}{\tt jean-paul.allouche@imj-prg.fr} \\
\and
G.-N. Han\\
CNRS, IRMA\\
Universit\'e de Strasbourg\\
7 rue Ren\'e Descartes \\
F-67084 Strasbourg France\\
\href{mailto:guoniu.han@unistra.fr}{\tt guoniu.han@unistra.fr} \\
\and
J. Shallit \\
 School of Computer Science \\
 University of Waterloo \\
Waterloo, ON  N2L 3G1 \\
Canada \\
\href{mailto:shallit@uwaterloo.ca}{\tt shallit@uwaterloo.ca} 
}

\title{On some conjectures of P. Barry}

\date{June 24, 2020}

\begin{document}

\maketitle

\begin{abstract}
We prove a number of conjectures recently stated by P. Barry, related
to the paperfolding sequence and the Rueppel sequence.
\end{abstract}



\section{Introduction}
\label{intro}
In his recent paper~\cite{barry}, P. Barry studied a number of integer 
sequences---in particular, the (regular) paperfolding sequence and the Rueppel sequence. 
Recall that the sequence $(j_n)_{n \geq 0}$ is defined by $j_0=0$, and, for $n > 0$, by the Jacobi-Kronecker
symbol $j_n = (\frac{-1}{n})$. It is sequence \seqnum{A034947} in the OEIS \cite{oeis}, and is, up to the 
first term, the $\pm 1$ {\it paperfolding sequence}, as noted by J. Sondow and recalled in \cite{barry}. 
It can thus be also defined by the relations 
$$
j_0 = 0, \ j_{2k}= j_k \ \mbox{\rm for all $k \geq 0$, and} \ j_{2k+1} = (-1)^k \ \mbox{\rm for all 
$k \geq 0$}.
$$
or the generating function
$$
\sum_{n\geq 0}j_n x^n = \sum_{k\geq 0} \frac{x^{2^k}}{1+x^{2^{k+1}}}.
$$

The {\it Rueppel sequence} $(r_n)_{n \geq 0} =
(1,1,0,1,0,0,0,1,\ldots)$ is the characteristic sequence of the set
$\{ 2^n - 1 \,:\, n \geq 0 \}$.  Alternatively, it is defined by
the generating function
$$
r(x)=\sum_{n\geq 0}r_n x^n= \sum_{n\geq 0} x^{2^n-1} = 1+x+x^3+x^7+x^{15}+\cdots
$$
One reason that P. Barry studied Rueppel sequence is its
relation with the famous Catalan sequence \cite{barry, Cigler2019, Deutsch2006Sagan}
$$r_n \equiv C_n=\frac{1}{n+1}\binom{2n}{n} \pmod 2.$$

We identify a sequence ${\bf a}=(a_0, a_1, a_2, \ldots)$ and its generating function 
$f=f(x)=a_0+a_1x+a_2x^2+\cdots$. Usually, $a_0=1$.
For each $n\geq 1$ the {\it Hankel determinant\/}
of the series $f$ (or of the sequence $\bf a$)  is defined by 
\begin{equation}
H_n(f) := \left|
\begin{matrix}
 a_0 & a_{1} & \ldots & a_{n-1} \cr
a_{1} & a_{2} & \ldots & a_{n} \cr
 \vdots  & \vdots  & \ddots &
 \vdots  \cr
a_{n-1} & a_{n} & \ldots & a_{2n-2} 
\end{matrix} \right|.
\end{equation}
We let $H_0(f)=1$. The {\it sequence of the Hankel determinants} of $f$ is defined as follows:
$$H(f):=(H_0(f), H_1(f), H_2(f), H_3(f), \ldots).$$ 

In \cite{barry}
Barry states several 
conjectures related to the (regular) paperfolding sequence and
the Hankel determinants of the modified
Rueppel sequences. 
In Section~\ref{sec2} we prove the first three conjectures, related to the paperfolding sequence.
Then, we suggest some further questions in Section~\ref{sec3}.
In Sections~\ref{han1}--\ref{han3},
we study the Hankel determinants of $1-xr(x)$, $1+xr(x)$, and $r(x)/(r(x)-x)$,
and prove Barry's Conjectures 6, 7, 8, 9, 10, 11, and 16,
respectively.

\section{On the first three conjectures of Barry}
\label{sec2}
One important sequence $(s_n)_{n \geq 0}$ in \cite{barry} is
the sequence \seqnum{A088748} in OEIS \cite{oeis}  defined by 
\begin{equation}\label{def:sn}
s_n = 1 + \displaystyle\sum_{0 \leq k \leq n} j_k,
\end{equation}
or by the generating function
	\begin{equation}
		\sum_{n\geq 0}s_n x^n =
		\frac{1}{1-x} \left(1+\sum_{k\geq 0}\frac{x^{2^k}}{1+x^{2^{k+1}}} \right).
\end{equation}

\begin{proposition}\label{prop-conj1}
The sequence $(s_n)_{n \geq 0}$ satisfies, for all $n \geq 0$, the relations
$$
\begin{aligned}
s_{2n} &= s_n + 
               \left\{ \begin{aligned}
                       &0, \ \mbox{\rm if $n$ is even;} \\ 
                       &1, \ \mbox{\rm if $n$ is odd.} \\ 
                       \end{aligned}
             \right. \\
s_{2n+1} &= s_n + 
               \left\{ \begin{aligned}
                       &1, \ \mbox{\rm if $n$ is even;} \\ 
                       &0, \ \mbox{\rm if $n$ is odd.} \\ 
                       \end{aligned}
             \right. \\            
\end{aligned}
$$
\end{proposition}

\begin{proof}
We write
$$
s_{2n} = 1 + \sum_{0 \leq k \leq 2n} j_k 
= 1 + \sum_{0 \leq m \leq n} j_{2m} + \sum_{0 \leq m \leq n-1} j_{2m+1}
= 1 + \sum_{0 \leq m \leq n} j_m + \sum_{0 \leq m \leq n-1} (-1)^m. 
$$
Hence
$$
s_{2n} = s_n + \sum_{0 \leq m \leq n-1} (-1)^m 
= s_n + \left\{ \begin{aligned}
                       &0, \ \mbox{\rm if $n$ is even;} \\ 
                       &1, \ \mbox{\rm if $n$ is odd.} \\ 
                       \end{aligned}
             \right.                     
$$
This implies
$$
s_{2n+1} = s_{2n} + j_{2n+1} = s_{2n} + (-1)^n 
= s_n + \left\{ \begin{aligned}
                       &1, \ \mbox{\rm if $n$ is even;} \\ 
                       &0, \ \mbox{\rm if $n$ is odd.} \\ 
                       \end{aligned}
             \right.
$$            
\end{proof}

\begin{theorem}\label{conj1}
Barry's Conjecture 1 is true: the locations of the occurrences of $m$ in the sequence 
$(s_n)_{n \geq 0}$ are given by those numbers whose base-$2$ representation has 
exactly $m-1$ runs. Furthermore the values of $m$ occurring in $\{s_0, s_1, \ldots, s_{2^N-1}\}$ 
are ${1, 2, \ldots, N+1}$.
\end{theorem}

\begin{proof}
To prove that the statement of the theorem holds for all indices $n$ of $(s_n)_{n \geq 0}$, 
we prove by induction on $N$ that the property holds for all $m$ occurring at indices 
$n \in [0, 2^N-1]$ of $(s_n)_{n \geq 0}$. 

The claim holds for $N=0$, since the only index 
to consider is then $n=0$, and $s_0 = 1$, while the binary expansion of $0$ is empty, and hence 
has no runs. 

Suppose that the desired property holds for all $n \in [0, 2^N-1]$.  Every number 
$n \in [0, 2^{N+1}-1]$ can be written $n = 2a + r$ where $r \in \{0, 1\}$. Thus $a$ is necessarily 
in $[0, 2^N-1]$. There are four cases. We have, using Proposition~\ref{prop-conj1}:
\begin{itemize}

\item if $r = 0$ and $a$ even, say $a = 2b$, then $n = 4b$. Thus $s_n = s_{4b} = s_{2b}= s_a$.
Since $n=4b$ and $a=2b$ have the same number of runs, the property holds for $n$.

\item if $r = 0$ and $a$ odd, say $a = 2b+1$, then $n=4b+2$. Thus $s_n = s_{4b+1} = s_{2b} +1$.
Since $n=4b+2$ clearly has one more run then $a=2b+1$, the property holds for $n$.

\item if $r = 1$ and $a$ even, say $a = 2b$, then $n = 4b+1$. Thus $s_n = s_{4b+1} = s_{2b} + 1$.
Since $n=4b+1$ has one more run than $a=2b$, the property holds for $n$.

\item  if $r = 1$ and $a$ odd, say $a = 2b+1$, then $n = 4b+3$. Thus $s_n = s_{4b+3} = s_{2b+1}$.
Since  $n = 4b+3$ and $a = 2b+1$ have the same number of runs, the property holds for $n$.

\end{itemize}
This completes the proof.
\end{proof}

\begin{remark}\label{adamson}
An alternative statement of Theorem~\ref{conj1} is that the sum $\sum_{0 \leq k \leq n} j_k$
is equal to the number of runs in the binary expansion of $n$: this was noted by G. W. Adamson 
in a 2008 comment on \seqnum{A005811} in the OEIS \cite{oeis}.
\end{remark}

\bigskip

Now we address two more conjectures of Barry in \cite{barry}. First we prove a general statement.

\begin{proposition}\label{general}
Let $(c_n)_{n \geq 0}$ be an increasing sequence of integers. Let $(\lambda_n)_{n \geq 0}$ be the 
characteristic function of the set $\{c_0, c_1, \ldots, c_n, \ldots\}$.  Then
$$
\forall n \geq 0, \ \ c_n - \sum_{0 \leq k \leq c_n} (-1)^{\lambda_k} = 2n + 1.
$$
\end{proposition}

\begin{proof}
First we note the equivalence
$$
c_n = r \ \Longleftrightarrow \ \left(\lambda_r = 1 \ \ 
\mbox{\rm and} \ \ \sum_{0 \leq k \leq r} \lambda_k = n+1 \right).
$$
But $(-1)^{\lambda_k} = 1 - 2\lambda_k$. Hence if $c_n = r$, 
then $\displaystyle\sum_{0 \leq k \leq r} \lambda_k = n+1 $. Thus
$$
\sum_{0 \leq k \leq c_n} (-1)^{\lambda_k} = \sum_{0 \leq k \leq r} (-1)^{\lambda_k} =
\sum_{0 \leq k \leq r} (1 - 2\lambda_k) =
r + 1 - 2 \displaystyle\sum_{0 \leq k \leq r} \lambda_k = r + 1 - (2n+2) = c_n - 2n -1.
$$
\end{proof}

\begin{remark}
Proposition~\ref{general} is related to the problem of computing the index $n$ of the
$n$th term of an increasing sequence of integers. In this direction the reader can consult
\cite{lambek-moser}.
\end{remark}

\medskip

\begin{theorem}\label{conj2}
Barry's Conjecture 2 is true.  Namely, define $a_0 = 0$, and
let $(a_n)_{n \geq 1}$ denote the increasing sequence of positive integers 
whose odd part is of the form $4k + 1$. Then we have $a_n + s_{a_n} = 2n + 1$.
\end{theorem}

\begin{proof}
We note that the recursive definition of $j_n$ implies that the sequence $(a_n)_{n \geq 0}$ is 
exactly the sequence of integers consisting of  $0$ and the integers $m \geq 1$ such that 
$j_m = 1$. In other words, the characteristic function of the set $\{a_0, a_1, a_2, \ldots, a_n, \ldots\}$ 
is the sequence $(\lambda_n)_{n \geq 0}$ with $\lambda_0 = 1$ and $\lambda_n = (1 + j_n)/2$ for 
${n \geq 1}$. So  $(-1)^{\lambda_0} = -1$ and $(-1)^{\lambda_k} = - j_k$ for $k \geq 1$. Thus 
Proposition~\ref{general} above with $(c_n)_{n \geq 0} = (a_n)_{n \geq 0}$ yields
$$
a_n + s_{a_n} = a_n + 1 + \sum_{1 \leq k \leq a_n} j_k
 = a_n - \sum_{0 \leq k \leq a_n} (-1)^{\lambda_k} = 2n+1.
$$
\end{proof}

\begin{theorem}\label{conj3}
Barry's Conjecture 3 is true. Let $(b_n)_{n \geq 0}$ denote the increasing sequence of integers 
whose odd part is of the form $4k + 3$. Then we have $b_n - s_{b_n} = 2n + 1$.
\end{theorem}
\begin{proof}
The integers in the sequence $(b_n)_{n \geq 0}$ are exactly the integers $m \geq 0$ for which 
$j_m = -1$. Hence the characteristic function of $(c_n)_{n \geq 1}$ is $(\mu_n)_{n \geq 0}$ 
with $\mu_0 = 0$ and $\mu_n = (1 - j_n)/2$ for ${n \geq 1}$. So  
$(-1)^{\mu_0} = 1$ and $(-1)^{\mu_k} = j_k$ for $k \geq 1$. Now we apply 
Proposition~\ref{general} with $(c_n)_{n \geq 0} $= $(b_n)_{n \geq 0}$ to get
$$
b_n - s_{b_n} = b_n - 1 - \sum_{1 \leq k \leq b_n}  j_k 
= b_n - \sum_{0 \leq k \leq a_n} (-1)^{\mu_k} = 2n+1.
$$ 
\end{proof}

We conclude this section with a table giving the first
few values of the sequences we have discussed.
\begin{center}
    \begin{tabular}{|c|ccccccccccccccccc}
    $n$ & 0 & 1 & 2 & 3 & 4 & 5 & 6 & 7 & 8 & 9 & 10 & 11 & 12 & 13 & 14 & 15 & 16 \\
    \hline
    $j_n$ & 0& 1& 1&$-1$& 1& 1&$-1$&$-1$& 1& 1& 1&$-1$&$-1$& 1&$-1$&$-1$& 1\\
    $s_n$& 1&2&3&2&3&4&3&2&3&4&5&4&3&4&3&2&3\\
    $a_n$ & 0& 1& 2& 4& 5& 8& 9&10&13&16&17&18&20&21&25&26&29\\
    $b_n$ & 3& 6& 7&11&12&14&15&19&22&23&24&27&28&30&31&35&38
    \end{tabular}
\end{center}

\section{A (non-)regularity property of the sequence $(a_n)_{n \geq 0}$}
\label{sec3}

While writing this paper, we realized that a result similar to Theorem~\ref{conj2} and to 
Theorem~\ref{conj3} was proved in \cite{abs}, where the paperfolding sequence was replaced 
with a generalized Thue-Morse sequence. For example, let $(t_n)_{n \geq 0}$ be the usual
Thue-Morse sequence (see, e.g., \cite{ubiq}), where $t_n$ is the parity of the sum of the binary 
digits of $n$. Let $(u_n)_{n \geq 0}$ denote the increasing sequence of {\it odious numbers}, 
namely the numbers $n$ for which $t_n = 1$, and $(v_n)_{n \geq 0}$ be the increasing sequence 
of {\it evil numbers}, namely the numbers $n$ for which $t_n = 0$. 
Then (see \cite[Corollary~1, p. 34]{abs}) we have
\begin{equation}
\forall n \geq 0, \ u_n = 2n + 1 - t_n, \ \ \mbox{\rm and} \ \ v_n = 2n + t_n. 
\label{onestar}
\end{equation}
This property can be compared with Proposition~\ref{general}. Namely, for all $n \geq 0$, we have
$t_{2n} = t_n$ and $t_{2n+1} = 1 - t_n$. Thus 
$$
\sum_{0 \leq k \leq n} (-1)^{t_k} = 
\begin{cases}
0,             &\mbox{\rm if $n$ is odd;} \\
(-1)^{t_n}, &\mbox{\rm if $n$ is even.}
\end{cases}
$$
Proposition~\ref{general} thus implies that (noting that for all $n$ one has $t_{u_n} = 1$)
$$
\forall n \geq 0, \  2n + 1 = u_n - \sum_{0 \leq k \leq u_n} (-1)^{t_k} = 
\begin{cases}
u_n, &\mbox{\rm if $u_n$ is odd;} \\
u_n - (-1)^{t_{u_n}} = u_n + 1, &\mbox{\rm if $u_n$ is even.}
\end{cases}
$$
This only says that $\lfloor \frac{u_n}{2}\rfloor = n$ (and similarly $\lfloor \frac{v_n}{2}\rfloor = n$),
which is clearly implied by, but weaker than, the identities in \eqref{onestar}, which state 
that $u_n = 1 - t_n \bmod 2$ and $v_n = t_n \bmod 2$. 

\bigskip

Recall that the {\em $r$-kernel} of a sequence $(x_n)_{n \geq 0} \in \Zee^\Enn$, for an integer 
$d \geq 2$, is the set of subsequences 
$$\{(x_{r^i n + j})_{n \geq 0}, \ i \geq 0, \ j \in [0, r^i -1]\}.$$ 
Also recall that a sequence $(x_n)_{n \geq 0}$ is called {\em $r$-automatic} if its $r$-kernel is finite, 
and $r$-regular if its $r$-kernel generates a $\Zee$-module of finite type (see \cite{AS}).

What precedes implies, in particular, that the sequences $(u_n)_{n \geq 0}$ and $(v_n)_{n \geq 0}$
are $2$-regular, while it is well-known that the Thue-Morse sequence $(t_n)_{n \geq 0}$ is 
$2$-automatic. Also note the asymptotic behavior $u_n \sim v_n \sim 2n$, for $n \to \infty$.

On the other hand the sequence $(j_n)_{n \geq 0}$, being equal to the $\pm 1$-paperfolding 
sequence up to its first term, is $2$-automatic. Also, we clearly have, from Remark~\ref{adamson}, 
that $s_n$ is at most one plus half the number of digits of $n$.   Thus, using Theorem~\ref{conj1}, 
we see that $s_n = {\mathcal O}(\log n)$, and so $a_n \sim 2n$. Hence $a_n = 2n + {\mathcal O}(\log a_n) = 
2n + {\mathcal O}(\log n)$. Similarly, using Theorem~\ref{conj2}, we get $b_n = 2n + {\mathcal O}(\log n)$. 

A question that now comes to mind is whether $(a_n)_{n \geq 0}$ and $(b_n)_{n \geq 0}$ are
$2$-regular sequences. We prove that they are not. First, we prove a proposition characterizing
the integers $n$ for which $a_{n+1} - a_n = 2$.

\begin{proposition}\label{16n+2}
\leavevmode
\begin{enumerate}[(i)]
\item
We have $a_{n+1} - a_n = 2$ for some integer $n$ if and only if there exists an integer $r$
such that $n = 8r + \frac{1 + s_{16r+2}}{2}$. Then $a_n = 16r +2$ and such an $r$ is unique,
with $r = \lfloor \frac{a_n}{16} \rfloor$.

\item Let denote $\psi(r) = 8r + \frac{1 + s_{16r+2}}{2}$, so that $a_{\psi(r)} = 16r + 2$ and
$r  = \lfloor \frac{a_{\psi(r)}}{16} \rfloor$. The function $\psi$ is increasing for $r \geq 0$. Furthermore,
for all $r, r'$ with $r \neq r'$ we have $|\psi(r) - \psi(r')| \geq 7$.
\end{enumerate}
\end{proposition}

\begin{proof}
\leavevmode
\begin{enumerate}[(i)]
\item First we note that $\frac{1 + s_{16r+2}}{2}$ is always an integer: the number of runs of an even
number is even, and hence $s(m)$ must be odd when $m$ is an even integer. 

Now let $n$ be such
that $a_{n+1} - a_n = 2$. Write $a_n = 4k+j$ with $j \in [0,3]$. Let $A$ denote the set of values of 
the sequence $(a_n)_{n \geq 0}$:  $A = \{0, 1, 2, 4, 5, 8, 9, 10, 13, 16, 17, 18, 20, 21,  \ldots\}$ 
We have that $j \neq 0$, for otherwise $a_n + 1 = 4k+1 \in A$ and $a_{n+1} = a_n +1$. We also have
$j \neq 1$, for otherwise $a_n + 2 \equiv \modd{3} {4}$ does not belong to $A$, and hence cannot equal 
$a_{n+2}$. Since $j \neq 3$ because $4k+3 \notin A$, we must have $a_n = 4k + 2$. Since
$4k+2$ belongs to $A$ if and only if $2k+1$ belongs to $A$, if and only if $k$ is even, there exists
$\ell$ with $k = 2\ell$. But then $a_{n+1} = a_n + 2 = 4k + 4 = 8 \ell + 4$ belongs to $A$: this is
equivalent to saying that $2 \ell +1$ belongs to $A$, which holds if and only if $\ell$ is even, say
$\ell = 2r$. Thus $a_n = 4k + 2 = 16 r + 2$. 
Theorem~\ref{conj2} then implies that
$16 r + 2 + s_{16n+2} = 2n + 1$. Hence $n = 8r + \frac{1 + s_{16r+2}}{2}$. 

Conversely, suppose
that there exists an integer $r$ such that $n = 8r + \frac{1 + s_{16r+2}}{2}$. The number $16n+r$
belongs to $A$, and hence there exists an integer $m$ with $a_m = 16 r + 2$. Since $a_m + 1 = 16r+3$
does not belong to $A$, and $a_m + 2 = 16r + 4 = 4(4r+1)$ belongs to $A$, we have that 
$a_{m+1} = a_m + 2$. Hence $a_{m+1} - a_m = 2$. To finish the proof, we claim that $m = n$.
Using the first part of the proof, we have that $m = 8r + \frac{1 + s_{16r+2}}{2}$. Hence $m = n$.

\item To prove that $\psi$ is increasing and that $|\psi(r) - \psi(r')| \geq 7$, it suffices to prove the 
inequality $\psi(r+1) - \psi(r) \geq 7$ for all $r \geq 0$. Thus it suffices to prove that 
$8 + \frac{s_{16r+18} - s_{16r+2}}{2} \geq 7$, for all $r \geq 0$. This last inequality would be 
implied by the inequality $|s_{16r+18} - s_{16r+2}| \leq 2$. But the properties of $s_n$ given in 
Proposition~\ref{prop-conj1} can be rewritten as follows:
$$
\left\{
\begin{aligned}
s_{4n}    &= s_{2n} \\
s_{4n+1} &= s_{2n}  + 1 \\
s_{4n+2} &= s_{2n+1} +1 \\
s_{4n+3} &= s_{2n+1}. \\
\end{aligned}
\right.
$$
Therefore $s_{16n+2} = s_{8n+1} + 1 = s_{4n} + 2 = s_{2n} + 2$. 
Hence 
$$s_{16r+18} - s_{16r+2} = s_{16(r+1)+2} - s_{16r+2} = s_{2r+2}- s_{2r}.$$ 
We distinguish two
cases according to the parity of $r$:
$$
\begin{aligned}
\mbox{\rm If $r = 2t$, then}     \ \ &s_{2r+2} - s_{2r} = s_{4t+2} - s_{4t} = s_{2t+1} +1 - s_{2t} 
= s_{r+1} - s_r + 1; \\
\mbox{\rm If $r = 2t+1$, then} \ \ &s_{2r+2} - s_{2r} = s_{4t+4} - s_{4t+2} = s_{2t+2} - s_{2t+1} - 1
= s_{r+1} - s_r - 1. \\
\end{aligned}
$$
Now it suffices to prove that $|s_{n+1} - s_n| \leq 1$, for all $n \geq 0$. But $s_n$ is equal to $1+$ 
the number of runs in the binary expansion of $n$ (see Remark~\ref{adamson} above). And it is easy
to see that the absolute value of the difference of the number of runs for two consecutive integers is 
equal to $1$ (see, e.g., the remark before Theorem~4 in \cite{shallit-norgard}).
\end{enumerate}
\end{proof}

\begin{theorem}\label{not-reg}
The sequence $(a_n)_{n \geq 0}$ is not $q$-regular for any $q \geq 2$.
\end{theorem}

\begin{proof}
If the sequence $(a_n)_{n \geq 0}$ were $q$-regular, its difference 
sequence, say $(d_n)_{n \geq 0}$, where $d_n = a_{n+1} - a_n$, would be 
$q$-regular as well. But $(d_n)_{n \geq 0}$ takes only finitely
many values: recall that all integers congruent to $1$ modulo $4$ are values of $a_n$, so that
$d_n \in [1, 4]$. Thus, if $(d_n)_{n \geq 0}$ were $q$-regular, it would be $q$-automatic.
The proof that the sequence $(d_n)_{n \geq 0}$ is not $q$-automatic is 
given in the next three theorems: we first prove that $(d_n)_{n \geq 0}$ 
is not $2$-automatic. Next, we prove that it is morphic.  Finally, we prove that it is not
$q$-automatic for any $q \geq 2$.
\end{proof}

\begin{theorem}\label{not-automatic}
The sequence $(d_n)_{n \geq 0}$ is not $2$-automatic.
\end{theorem}

\begin{proof}
Since a $2$-automatic sequence 
$(x_n)_{n \geq 0}$ is characterized by the fact that its $2$-kernel (i.e., the set of subsequences 
$\{(x_{2^k n + j})_{n \geq 0}, \ k \geq 0, \ j \in [0, 2^k - 1]\}$) is finite, it suffices to prove that the 
subsequences $(u_{\alpha}(n))_{n \geq 0}$ and $(u_{\beta}(n))_{n \geq 0}$ are distinct, where 
$u_{\alpha}(n) := d(2^{2^{\alpha}} n + 2^{\alpha})$ and $u_{\beta}(n) = d(2^{2^{\beta}} n + 2^{\beta})$  
with $\alpha < \beta$ and $\alpha$ large enough. First note that
$$
t = \underbrace{1010 \cdots 10}_{\substack{j \ \mbox{\rm\scriptsize blocks} \ 10}} \ 
\mbox{\rm (in base $2$)} \ \Longrightarrow \ t = \frac{2}{3}(2^{2j} - 1) \ \mbox{\rm and} \  s(t) = 2j+1.
$$
Now take $j = 2^{\alpha} - 2$ and $t = \frac{2}{3} (2^{2^{{\alpha} + 1} - 4} - 1)$, so that 
$s(t) = 2^{\alpha+1} - 3$. Define $r = 2^{2^{\alpha} - 3} t$. 
Then $s(16r + 2) = s(2^{2^{\alpha} + 1} t + 2) = 2 + s(2t) = 2 + s(t)$ (recall that $t$ is even). 
So $s(16r + 2) = 2 + 2^{{\alpha}+1} - 3 =  2^{{\alpha}+1} - 1$. 
Hence $\psi(r) = 8r + 2^{\alpha} = 2^{2^{\alpha}} t + 2^{\alpha}$. 
We deduce, using Proposition~\ref{16n+2}, that $u_{\alpha}(t) = 2$. 

Now we prove that $u_{\beta}(t) \neq 2$. Define $r' = 2^{2^{\beta} - 3} t + 2^{\beta - 3} - 2^{\alpha - 3}$.
Then 
\begin{align*}
s(16 r' + 2) &= s(2^{2^{\beta} + 1} t + 2^{\beta + 1} - 2^{\alpha + 1} + 2) \\
&= 
s(2t) + s(2^{\beta + 1} - 2^{\alpha + 1} + 2) \\
& = s(2t) + 4 = s(t) + 4 = 2^{\alpha + 1} + 1.
\end{align*}
Hence $\psi(r') = 8r' + \frac{1 + s(16 r' + 2)}{2} = 2^{2^{\beta}} t + 2^{\beta} + 1$.
Thus $\psi(r') - (2^{2^{\beta}} t + 2^{\beta}) =  1$. Since $|\psi(r'') - \psi(r')| \geq 7$ for all $r''\neq r'$
(Proposition~\ref{16n+2}), the integer $(2^{2^{\beta}} t + 2^{\beta})$ cannot be equal to some
$\psi(r'')$. Hence $u_{\beta}(t) = a(2^{2^{\beta}} t + 2^{\beta}) \neq 2$ (again from 
Proposition~\ref{16n+2}). Hence $u_{\beta}(t) \neq u_{\alpha}(t)$, so that the sequences $u_{\beta}$ 
and $u_{\alpha}$ are distinct. A similar proof gives that $(b_n)_{n \geq 0}$ is not $2$-regular.
\end{proof}

\medskip

The next theorem proves that the sequence $(d_n)_{n \geq 0}$ is morphic.
For more about morphic sequences, see, e.g., \cite{AS}.

\begin{theorem}\label{morphic}
The sequence ${\mathbf d} = (d_n)_{n \geq 0}$ is morphic. More precisely, 
let ${\mathbf d'} = (d_{n+1})_{n \geq 0}$. If we define the morphisms $f$
on $\{0,1,2,3\}^*$ and $g$ from $\{0,1,2,3\}^*$ to $\{1,2,3,4\}^*$ as follows:
$$
f(0) = 01, \ f(1) = 21, \ f(2) = 03, \ f(3) = 23, 
$$
and
$$
g(0) = 121, \ g(1) = 31, \ g(2) = 13, \ g(3) = 4.
$$
Then \ ${\mathbf d'} = g(f^{\infty}(0))$.  That is, $\displaystyle{\mathbf d'} = \lim_{n \to \infty} g(f^{(n)}(0)$.
\end{theorem}

\begin{proof}
We have that ${\mathbf d}$ is morphic if and only if ${\mathbf d'}$ is. The characteristic sequence of
$(a_n)_{n \geq 0}$ is the paperfolding sequence $(p_n)_{n \geq 0} = (j_{n+1})_{n \geq 0}$ (with the 
notation above). We recall that the paperfolding sequence can be defined using ``perturbed 
symmetry'' as follows: let $X_k = p_0 p_1 \cdots p_k$ be its prefix of length $k+1$. Then $X_0 = 1$ 
and for all $k \geq 0$, $X_{k+1} = X_k \ 1 \ \overline{X_k^R}$, where $W^R$ is the word obtained from 
$W$ by writing it backwards, and $\overline{W}$ is the word obtained from $W$ by replacing $0$'s 
with $1$'s and $1$'s with $0$'s (see, e.g., \cite{mendes1} or \cite{mendes2}). Now let us define $U_k$ 
(resp., $V_k$) to be the word of distances between consecutive $1$'s (resp., consecutive $0$'s) in 
$X_k$. For example
$$
\begin{array}{llll}
&X_0 = 1               &U_0 = \epsilon      &V_0 = \epsilon \\
&X_1 = 110           &U_1 = 0                    &V_1 = \epsilon \\
&X_2 = 1101100   &U_2 = 010                &V_2 = 20 \\
&X_3 = 110110011100100 \ \ \ &U_3 = 0102002   \ \ &V_3 = 203010 \\
&\quad\quad \vdots & \quad\quad \vdots & \quad\quad \vdots
\end{array}
$$
The perturbed symmetry definition of the $X_k$'s, and the fact that, for $k \geq 2$, $X_k$ begins with
$11$ and ends with $00$, show that
$$
U_{k+1} = U_k \ 2 \ 0 \ V_k^R \ \ \mbox{\rm and} \ \ V_{k+1} = V_k \ 3 \ U_k^R.
$$
The sequence of words $(U_k)_k$ (resp., $(V_k)_k$) clearly converges to an infinite sequence $U$ 
(resp., $V$) with values in $\{0, 1, 2, 3\}$. It is straightforward to see that the difference between the
indexes of the $1$'s (resp., the $0$'s) in the sequence $(p_n)_{n \geq 0}$ are obtained by adding 
$1$ to the terms of the sequence $U$ (resp., the sequence $W$). Thus, the difference of 
indexes of consecutive $1$'s (resp., consecutive $0$'s) in $(p_n)_{n \geq 0}$ (which is the 
difference sequence of $(a_{n+1})_{n \geq 0}$) is given by the sequence $A$ 
(resp., $B$) on 
$\{1, 2, 3, 4\}$ which is the limit of the sequence of words $(A_k)_k$ 
(resp., $(B_k)_k$) defined, 
for $k \geq 0$, as follows:
$$
A_0 =  1 \ 2 \ 1, \ B_0 = 3 \ 1, \ \mbox{\rm and, for all $k \geq 0$,} \ 
A_{k+1} = A_k \ 3 \ 1 \ B_k^R \ \ \mbox{\rm and} \ \ B_{k+1} = B_k \ 4 \ A_k^R.
$$
If we prove that, for all $k \geq 2$, one has
\begin{equation}
g(f^{(k)}(0)) = A_k
\label{twostar}
\end{equation} we obtain, by letting $k$ 
tend to infinity, that $A = g(f^{\infty}(0))$. To obtain \eqref{twostar} we prove by induction on $k \geq 1$ that
\begin{equation}
A_k = g(f^{(k)}(0)) \ B_{k-1}^R, \ \ \  g(f^{(k)}(1)) = B_{k-1}^R \ 3 \ 1, \ \ \
g(f^{(k)}(2)) = A_{k-1} \ 4, \ \ \ g(f^{(k)}(3)) = B_{k-1}^R \ 4. 
\label{threestar}
\end{equation}
First we check \eqref{threestar} for $k=1$:
$$
\begin{array}{ll}
&g(f(0)) \ B_0^R = g(01) \ 1 \ 3 = g(0) \ g(1) \ 1 \ 3 = 1 \ 2 \ 1 \ 3 \ 1 \ 1 \ 3 = 
A_0 \ 3 \ 1 \ B_0^R  =  A_1 \\
&g(f(1)) = g(2 \ 1) = 1 \ 3 \ 3 \ 1 = B_0^R \ 3 \ 1 \\
&g(f(2)) = g(0 \ 3) = 1 \ 2 \ 1 \ 4 = A_0 \ 4 \\
&g(f(3)) = g(2 \ 3) = 1 \ 3 \ 4 = B_0^R \ 4. \\
\end{array}
$$
Now suppose that \eqref{threestar} holds for $k \geq 1$. Then
$$
\begin{aligned}
g(f^{(k+1)}(0)) \ B_k^R &= g(f^{(k)}(f(0)) \ B_k^R = g(f^{(k)}(0 \ 1)) \ B_k^R = 
g(f^{(k)}(0))  \ g(f^{(k)}(1))  \ B_k^R
\\
&= g(f^{(k)}(0)) \ B_{k-1}^R  \ 3  \ 1  \ B_k^R = A_k  \ 3 \ 1 \ B_k^R = A_{k+1};
\end{aligned}
$$
$$
\begin{aligned}
g(f^{(k+1)}(1)) &= g(f^{(k)}(f(1)) = g(f^{(k)}(2 \ 1)) = g(f^{(k)}(2)) \ g(f^{(k)}(1)) = 
A_{k-1} \ 4 \ B_{k-1}^R \ 3 \ 1
\\
&= (B_{k-1} \ 4 \ A_{k-1}^R)^R \ 3 \ 1 = B_k^R \ 3 \ 1;
\end{aligned}
$$
$$
\begin{aligned}
g(f^{(k+1)}(2)) &= g(f^{(k)}(f(2)) = g(f^{(k)}(0 \ 3)) = g(f^{(k)}(0)) \ g(f^{(k)}(3)) = 
g(f^{(k)}(0)) \ B_{k-1}^R \ 4 
\\
&= A_k \ 4;
\end{aligned}
$$
$$
\begin{aligned}
g(f^{(k+1)}(3)) &= g(f^{(k)}(f(3)) = g(f^{(k)}(2 \ 3)) = g(f^{(k)}(2)) \ g(f^{(k)}(3)) = 
g(f^{(k)}(2)) \ B_{k-1}^R \ 4 
\\
&= A_{k-1} \ 4 \ B_{k-1}^R \ 4 = (B_{k-1} \ 4 \ A_{k-1}^R)^R \ 4 = B_k^R \ 4.
\end{aligned}
$$
\end{proof}

\begin{remark}
What we have proved in Theorem~\ref{morphic} above is that the sequence
${\mathbf d'}$ satisfies ${\mathbf d'} = g(f^{\infty}(0))$. This is not
exactly the definition of a morphic sequence, because $g$ is not a 
{\em coding}, i.e., a pointwise map, but a morphism. This is known 
to be equivalent to saying that ${\mathbf d'}$ is morphic 
(see, e.g., \cite{AS}). Note that it might be much easier 
to discover and prove that a sequence ${\mathbf u}$ is equal to, say, 
$\beta(\alpha^{\infty}(0))$ where $\beta$ and $\alpha$ are two morphisms, 
than to exhibit a morphism $\alpha_1$ and a coding $\varphi$ such that 
${\mathbf u} = \varphi(\alpha_1^{\infty}(0))$. In the first construction
the morphism $\beta$ is called in \cite{dekking} a {\em decoration} of the
fixed point $\alpha^{\infty}(0))$.
\end{remark}

Our last theorem proves that the sequence $(d_n)_{n \geq 0}$ is 
not $q$-automatic for any $q \geq 2$, thus finishing the proof of 
Theorem~\ref{not-reg}.

\begin{theorem}
The sequence $(d_n)_{n \geq 0}$ is not $q$-automatic for any $q \geq 2$.
\end{theorem}

\begin{proof}
We use a deep result of Durand \cite{durand} who widely generalized 
Cobham's theorem. Namely, from \cite[Corollary~6]{durand}, the sequence 
${\mathbf d'}$ is $2^k$-substitutive for some integer $k \geq 1$, since 
$f$ is a $2$-uniform morphism. So, if it were $q$-automatic for some 
$q \geq 2$, then either $q$ would be a power of $2$, or the sequence would
be ultimately periodic \cite[Theorem~1]{durand}. But both possibility
are ruled out by the fact that ${\mathbf d'}$ is not $2$-automatic from
Theorem~\ref{not-automatic}, and hence is neither $2^{\ell}$-automatic for any 
$\ell \geq 1$, nor ultimately periodic.
\end{proof}

\begin{remark}
The sequence $(b_n)_{n \geq 0}$ can be studied in a way similar to the
study of the sequence $(a_n)_{n \geq 0}$.
\end{remark}

\begin{remark}
It was already known that a sequence whose characteristic function is automatic is not necessarily 
regular. For example, Cateland \cite{cateland} studied the expansions of integers in base $q$ with 
digits in $\{d, d+1, \ldots, d+q-1\}$, for some $d \in [2 - q, 0]$. Using his results about integers that 
miss some digit(s) \cite[p.~90--105]{cateland}, one has the following:
{\it 
\begin{itemize}

\item Let $\{0, 1, 3, 4, 9, 10, 12, 13, \ldots\}$ be the increasing sequence of integers whose base-$3$ representation contains no $2$ (sequence \seqnum{A005836} in \cite{oeis}). Then the characteristic function
of the values of this sequence is $3$-automatic, while the sequence itself is $2$-regular, satisfying
$z_{2n} = 3 z_n$ and $z_{2n+1} = 3z_n + 1$.

\item Let $2, 8, 26, \ldots$ be the increasing sequence of integers whose base-$3$ expansion has 
all digits equal to $2$ (this is the increasing sequence $(3^n - 1)_{n \geq 0}$). The characteristic 
function of this sequence is $3$-automatic, while the sequence itself is not $r$-regular for all
$r \geq 2$ (note that it is the intersection of two $2$-regular sequences).

\end{itemize}
}
\end{remark}

What precedes leads to a general question.

\begin{question}
Let $(c_n)_{n \geq 0}$ be an increasing sequence of integers. Let 
$(\lambda_n)_{n \geq 0}$ be the characteristic sequence of the set 
$\{c_0, c_1, \ldots, c_n, \ldots\}$. Give a closed form or an asymptotic 
formula for $(c_n)$. Furthermore, if $(\lambda_n)_{n \geq 0}$ has some 
sort of regularity, does $(c_n)_{n \geq 0}$ inherit a ``similar'' 
regularity? 
In particular, if $(\lambda_n)_{n \geq 0}$ is a {\it $q$-automatic 
sequence}, when is it true that the sequence $(c_n)_{n \geq 0}$ is 
{\it $\ell$-regular} for some $\ell \geq 2$ (where, possibly, 
$\ell \neq q$)? 
\end{question}


\section{The Hankel determinants of $1-xr(x)$} 
\label{han1}
Let
\begin{equation}\label{eq:Bx}
	B(x)=1-xr(x)=1-\sum_{k\geq 0} x^{2^k} = 1-x-x^2-x^4-x^8-x^{16}-\cdots
\end{equation}
The first terms of the Hankel determinants $B(x)$ are
$$
H(B(x))= (1, 1, -2, 3, 2, -3, 4, 3, 2, -3, 4, -5, -4, -3, 4, 3, 2,\ldots)
$$

Consider the sequence $(-r_1, -r_2, -r_3, \ldots)$ obtained from $B(x)$ by shifting two times, i.e.,
$$
T(x)=\frac{B(x)-(1-x)}{x^2}=\frac{1-r}{x}.
$$
It is also the negative shifted Rueppel sequence.  Let $g_n=H_n(T(x))$. We establish
the following characterization of $g_n$.

\begin{lemma}\label{th:g:g}
	We have $g_0=1, g_1=-1$, and
	\begin{equation}\label{eq:g:g}	
	g_n=(-1)^{n+1} g_{2^{k+1}-n-1},
	\end{equation}	
	where $2^k<n+1\leq 2^{k+1}$.
\end{lemma}
\begin{proof}
First, Proposition~4 in \cite{barry}
implies
	that $g_0=1$, and
\begin{equation}\label{eq:g:prop4}
	g_{2n} = (-1)^{n(n+1)/2} g_n,
	\qquad
	g_{2n+1} = (-1)^{(n+1)(n+2)/2} g_n.
\end{equation}
	Next, we prove \eqref{eq:g:g} by induction on $n$ by using \eqref{eq:g:prop4}. 
	We can verify \eqref{eq:g:g} is true for $n=2,3$.
	Now, suppose that  \eqref{eq:g:g} is true for $n\leq 2m-1$ (with $m\geq 2$).
	We consider two cases.

\bigskip

\noindent {(i)} The case $n$ even, $n=2m$:
we need to prove that
\begin{equation}\label{eq:g:need1}
	g_{2m} = - g_{2^{k+1} - 2m-1}, \qquad \text{with \quad} 2^k<2m+1\leq 2^{k+1}.
\end{equation}
Since
	$2^k<2m+1\leq 2^{k+1}$ is equivalent to 
$2^k<2m+2\leq 2^{k+1}$
or
	$2^{k-1}<m+1\leq 2^{k}$,
	by the induction hypothesis we have
	$$
	g_m=(-1)^{m+1}g_{2^k-m-1}.
	$$
	By \eqref{eq:g:prop4}, the left hand side of \eqref{eq:g:need1} is equal to
	$$
	g_{2m}=	(-1)^{m(m+1)/2} g_m
	=
	(-1)^{m(m+1)/2 +m+1} g_{2^k-m-1}.
	$$
	Since $k\geq 2$,
	the right hand side of \eqref{eq:g:need1} is equal to
	$$
	-g_{2^{k+1}-2m-1}
	=
	(-1)^{(2^k-m)(2^k-m+1)/2+1} g_{2^k-m-1}
	=
	(-1)^{m(m-1)/2+1} g_{2^k-m-1}.
	$$
	Hence \eqref{eq:g:need1} is true.

\bigskip
	
\noindent {(ii)} The case $n$ odd, $n=2m+1$: 	we need to prove that
\begin{equation}\label{eq:g:need2}
	g_{2m+1} =  g_{2^{k+1} - 2m-2}, \qquad \text{with \quad} 2^k<2m+2\leq 2^{k+1}.
\end{equation}
Since
	$2^k<2m+2\leq 2^{k+1}$ is equivalent to 
	$2^{k-1}<m+1\leq 2^{k}$,
	by the induction hypothesis we have
	$$
	g_m=(-1)^{m+1}g_{2^k-m-1}.
	$$
	By \eqref{eq:g:prop4}, the two sides of \eqref{eq:g:need2} are equal to
	$$
	g_{2m+1}=	(-1)^{(m+1)(m+2)/2} g_m
	=
	(-1)^{(m+1)(m+2)/2 +m+1} g_{2^k-m-1},
	$$
	$$
	g_{2^{k+1}-2m-2}
	=
	(-1)^{(2^k-m-1)(2^k-m)/2} g_{2^k-m-1}
	=
	(-1)^{m(m+1)/2} g_{2^k-m-1}.
	$$
	Hence \eqref{eq:g:need2} is true.
\end{proof}

\begin{notation}
We recall the notation $\sign(y)$:
$$
\sign(y) =
\begin{cases}
+1 &\text{if $y \geq 0$} \\
-1 &\text{if $y <0$}. \\
\end{cases}
$$
\end{notation}

\bigskip

Let $h_n=H_n(B(x))$.
\begin{theorem}\label{th:main:h}
	We have $h_0=h_1=1,h_2=-2$, and for each $n\geq 3$,
\begin{equation}\label{eq:hn}
	h_n= (-1)^{n} (h_m + g_{m-1}),
\end{equation}
	where $2^k< n\leq 2^{k+1}$ and  $ m=2^{k+1}-n+1$.
\end{theorem}
\begin{proof}
Our proof is	by the fundamental properties of determinants. 
As illustrated in Figure~\ref{fig:det}
with the example $n=11, k=3, m=6$, we have
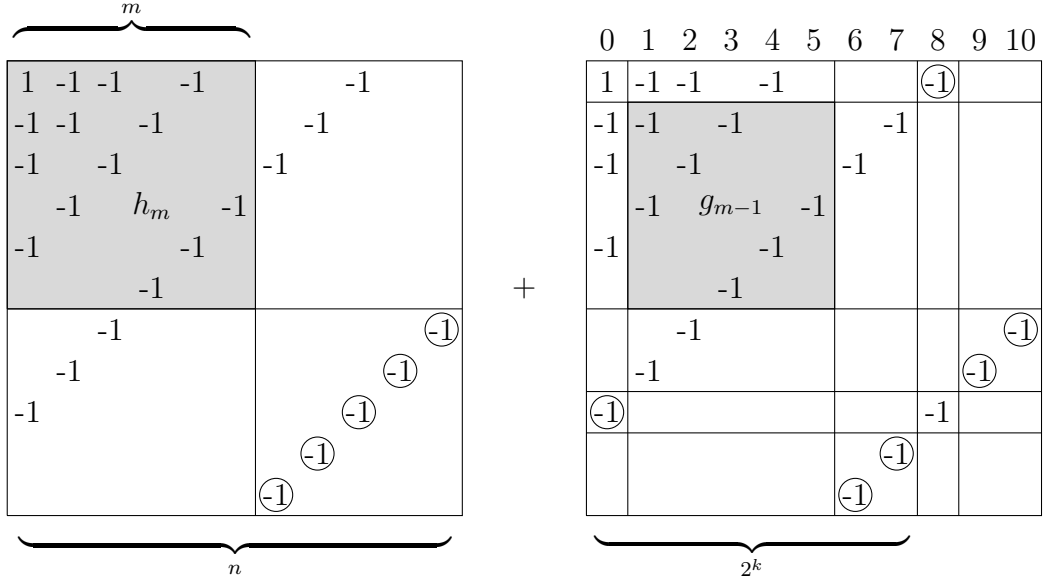
\begin{figure}[tbp]
\begin{center}
\begin{tikzpicture} 
\fill [gray!30, line width=0.4pt](0.0000,3.3000)--(3.3000,3.3000)--(3.3000,6.6000)--(0.0000,6.6000)--(0.0000,3.3000);
\draw [gray!200, line width=0.4pt](0.0000,3.3000)--(3.3000,3.3000)--(3.3000,6.6000)--(0.0000,6.6000)--(0.0000,3.3000);
\draw [gray!200] (0.2750, 6.3250) node [] {1}; 
\draw [gray!200] (0.8250, 6.3250) node [] {-1}; 
\draw [gray!200] (1.3750, 6.3250) node [] {-1}; 
\draw [gray!200] (1.9250, 6.3250) node [] {}; 
\draw [gray!200] (2.4750, 6.3250) node [] {-1}; 
\draw [gray!200] (3.0250, 6.3250) node [] {}; 
\draw [gray!200] (3.5750, 6.3250) node [] {}; 
\draw [gray!200] (4.1250, 6.3250) node [] {}; 
\draw [gray!200] (4.6750, 6.3250) node [] {-1}; 
\draw [gray!200] (5.2250, 6.3250) node [] {}; 
\draw [gray!200] (5.7750, 6.3250) node [] {}; 
\draw [gray!200] (0.2750, 5.7750) node [] {-1}; 
\draw [gray!200] (0.8250, 5.7750) node [] {-1}; 
\draw [gray!200] (1.3750, 5.7750) node [] {}; 
\draw [gray!200] (1.9250, 5.7750) node [] {-1}; 
\draw [gray!200] (2.4750, 5.7750) node [] {}; 
\draw [gray!200] (3.0250, 5.7750) node [] {}; 
\draw [gray!200] (3.5750, 5.7750) node [] {}; 
\draw [gray!200] (4.1250, 5.7750) node [] {-1}; 
\draw [gray!200] (4.6750, 5.7750) node [] {}; 
\draw [gray!200] (5.2250, 5.7750) node [] {}; 
\draw [gray!200] (5.7750, 5.7750) node [] {}; 
\draw [gray!200] (0.2750, 5.2250) node [] {-1}; 
\draw [gray!200] (0.8250, 5.2250) node [] {}; 
\draw [gray!200] (1.3750, 5.2250) node [] {-1}; 
\draw [gray!200] (1.9250, 5.2250) node [] {}; 
\draw [gray!200] (2.4750, 5.2250) node [] {}; 
\draw [gray!200] (3.0250, 5.2250) node [] {}; 
\draw [gray!200] (3.5750, 5.2250) node [] {-1}; 
\draw [gray!200] (4.1250, 5.2250) node [] {}; 
\draw [gray!200] (4.6750, 5.2250) node [] {}; 
\draw [gray!200] (5.2250, 5.2250) node [] {}; 
\draw [gray!200] (5.7750, 5.2250) node [] {}; 
\draw [gray!200] (0.2750, 4.6750) node [] {}; 
\draw [gray!200] (0.8250, 4.6750) node [] {-1}; 
\draw [gray!200] (1.3750, 4.6750) node [] {}; 
\draw [gray!200] (1.9250, 4.6750) node [] {}; 
\draw [gray!200] (2.4750, 4.6750) node [] {}; 
\draw [gray!200] (3.0250, 4.6750) node [] {-1}; 
\draw [gray!200] (3.5750, 4.6750) node [] {}; 
\draw [gray!200] (4.1250, 4.6750) node [] {}; 
\draw [gray!200] (4.6750, 4.6750) node [] {}; 
\draw [gray!200] (5.2250, 4.6750) node [] {}; 
\draw [gray!200] (5.7750, 4.6750) node [] {}; 
\draw [gray!200] (0.2750, 4.1250) node [] {-1}; 
\draw [gray!200] (0.8250, 4.1250) node [] {}; 
\draw [gray!200] (1.3750, 4.1250) node [] {}; 
\draw [gray!200] (1.9250, 4.1250) node [] {}; 
\draw [gray!200] (2.4750, 4.1250) node [] {-1}; 
\draw [gray!200] (3.0250, 4.1250) node [] {}; 
\draw [gray!200] (3.5750, 4.1250) node [] {}; 
\draw [gray!200] (4.1250, 4.1250) node [] {}; 
\draw [gray!200] (4.6750, 4.1250) node [] {}; 
\draw [gray!200] (5.2250, 4.1250) node [] {}; 
\draw [gray!200] (5.7750, 4.1250) node [] {}; 
\draw [gray!200] (0.2750, 3.5750) node [] {}; 
\draw [gray!200] (0.8250, 3.5750) node [] {}; 
\draw [gray!200] (1.3750, 3.5750) node [] {}; 
\draw [gray!200] (1.9250, 3.5750) node [] {-1}; 
\draw [gray!200] (2.4750, 3.5750) node [] {}; 
\draw [gray!200] (3.0250, 3.5750) node [] {}; 
\draw [gray!200] (3.5750, 3.5750) node [] {}; 
\draw [gray!200] (4.1250, 3.5750) node [] {}; 
\draw [gray!200] (4.6750, 3.5750) node [] {}; 
\draw [gray!200] (5.2250, 3.5750) node [] {}; 
\draw [gray!200] (5.7750, 3.5750) node [] {}; 
\draw [gray!200] (0.2750, 3.0250) node [] {}; 
\draw [gray!200] (0.8250, 3.0250) node [] {}; 
\draw [gray!200] (1.3750, 3.0250) node [] {-1}; 
\draw [gray!200] (1.9250, 3.0250) node [] {}; 
\draw [gray!200] (2.4750, 3.0250) node [] {}; 
\draw [gray!200] (3.0250, 3.0250) node [] {}; 
\draw [gray!200] (3.5750, 3.0250) node [] {}; 
\draw [gray!200] (4.1250, 3.0250) node [] {}; 
\draw [gray!200] (4.6750, 3.0250) node [] {}; 
\draw [gray!200] (5.2250, 3.0250) node [] {}; 
\draw [gray!200] (5.7750, 3.0250) node [] {-1}; 
\draw [gray!200] (0.2750, 2.4750) node [] {}; 
\draw [gray!200] (0.8250, 2.4750) node [] {-1}; 
\draw [gray!200] (1.3750, 2.4750) node [] {}; 
\draw [gray!200] (1.9250, 2.4750) node [] {}; 
\draw [gray!200] (2.4750, 2.4750) node [] {}; 
\draw [gray!200] (3.0250, 2.4750) node [] {}; 
\draw [gray!200] (3.5750, 2.4750) node [] {}; 
\draw [gray!200] (4.1250, 2.4750) node [] {}; 
\draw [gray!200] (4.6750, 2.4750) node [] {}; 
\draw [gray!200] (5.2250, 2.4750) node [] {-1}; 
\draw [gray!200] (5.7750, 2.4750) node [] {}; 
\draw [gray!200] (0.2750, 1.9250) node [] {-1}; 
\draw [gray!200] (0.8250, 1.9250) node [] {}; 
\draw [gray!200] (1.3750, 1.9250) node [] {}; 
\draw [gray!200] (1.9250, 1.9250) node [] {}; 
\draw [gray!200] (2.4750, 1.9250) node [] {}; 
\draw [gray!200] (3.0250, 1.9250) node [] {}; 
\draw [gray!200] (3.5750, 1.9250) node [] {}; 
\draw [gray!200] (4.1250, 1.9250) node [] {}; 
\draw [gray!200] (4.6750, 1.9250) node [] {-1}; 
\draw [gray!200] (5.2250, 1.9250) node [] {}; 
\draw [gray!200] (5.7750, 1.9250) node [] {}; 
\draw [gray!200] (0.2750, 1.3750) node [] {}; 
\draw [gray!200] (0.8250, 1.3750) node [] {}; 
\draw [gray!200] (1.3750, 1.3750) node [] {}; 
\draw [gray!200] (1.9250, 1.3750) node [] {}; 
\draw [gray!200] (2.4750, 1.3750) node [] {}; 
\draw [gray!200] (3.0250, 1.3750) node [] {}; 
\draw [gray!200] (3.5750, 1.3750) node [] {}; 
\draw [gray!200] (4.1250, 1.3750) node [] {-1}; 
\draw [gray!200] (4.6750, 1.3750) node [] {}; 
\draw [gray!200] (5.2250, 1.3750) node [] {}; 
\draw [gray!200] (5.7750, 1.3750) node [] {}; 
\draw [gray!200] (0.2750, 0.8250) node [] {}; 
\draw [gray!200] (0.8250, 0.8250) node [] {}; 
\draw [gray!200] (1.3750, 0.8250) node [] {}; 
\draw [gray!200] (1.9250, 0.8250) node [] {}; 
\draw [gray!200] (2.4750, 0.8250) node [] {}; 
\draw [gray!200] (3.0250, 0.8250) node [] {}; 
\draw [gray!200] (3.5750, 0.8250) node [] {-1}; 
\draw [gray!200] (4.1250, 0.8250) node [] {}; 
\draw [gray!200] (4.6750, 0.8250) node [] {}; 
\draw [gray!200] (5.2250, 0.8250) node [] {}; 
\draw [gray!200] (5.7750, 0.8250) node [] {}; 
\draw [gray!200, line width=0.4pt](0.0000,0.5500)--(6.0500,0.5500)--(6.0500,6.6000)--(0.0000,6.6000)--(0.0000,0.5500);
\draw [gray!200] (6.8750, 3.5750) node [] {$+$}; 
\draw [ gray!200, line width=0.4pt](3.5750,0.8250) circle (0.2200) ;
\draw [ gray!200, line width=0.4pt](4.1250,1.3750) circle (0.2200) ;
\draw [ gray!200, line width=0.4pt](4.6750,1.9250) circle (0.2200) ;
\draw [ gray!200, line width=0.4pt](5.2250,2.4750) circle (0.2200) ;
\draw [ gray!200, line width=0.4pt](5.7750,3.0250) circle (0.2200) ;
\draw [gray!200] (1.9250, 4.6750) node [] {$h_m$}; 
\draw [ gray!200, line width=0.4pt](0.0000,3.3000)--(6.0500,3.3000);
\draw [ gray!200, line width=0.4pt](3.3000,0.5500)--(3.3000,6.6000);
\fill [gray!30, line width=0.4pt](8.2500,3.3000)--(11.0000,3.3000)--(11.0000,6.0500)--(8.2500,6.0500)--(8.2500,3.3000);
\draw [gray!200, line width=0.4pt](8.2500,3.3000)--(11.0000,3.3000)--(11.0000,6.0500)--(8.2500,6.0500)--(8.2500,3.3000);
\draw [gray!200] (7.9750, 6.3250) node [] {1}; 
\draw [gray!200] (8.5250, 6.3250) node [] {-1}; 
\draw [gray!200] (9.0750, 6.3250) node [] {-1}; 
\draw [gray!200] (9.6250, 6.3250) node [] {}; 
\draw [gray!200] (10.1750, 6.3250) node [] {-1}; 
\draw [gray!200] (10.7250, 6.3250) node [] {}; 
\draw [gray!200] (11.2750, 6.3250) node [] {}; 
\draw [gray!200] (11.8250, 6.3250) node [] {}; 
\draw [gray!200] (12.3750, 6.3250) node [] {-1}; 
\draw [gray!200] (12.9250, 6.3250) node [] {}; 
\draw [gray!200] (13.4750, 6.3250) node [] {}; 
\draw [gray!200] (7.9750, 5.7750) node [] {-1}; 
\draw [gray!200] (8.5250, 5.7750) node [] {-1}; 
\draw [gray!200] (9.0750, 5.7750) node [] {}; 
\draw [gray!200] (9.6250, 5.7750) node [] {-1}; 
\draw [gray!200] (10.1750, 5.7750) node [] {}; 
\draw [gray!200] (10.7250, 5.7750) node [] {}; 
\draw [gray!200] (11.2750, 5.7750) node [] {}; 
\draw [gray!200] (11.8250, 5.7750) node [] {-1}; 
\draw [gray!200] (12.3750, 5.7750) node [] {}; 
\draw [gray!200] (12.9250, 5.7750) node [] {}; 
\draw [gray!200] (13.4750, 5.7750) node [] {}; 
\draw [gray!200] (7.9750, 5.2250) node [] {-1}; 
\draw [gray!200] (8.5250, 5.2250) node [] {}; 
\draw [gray!200] (9.0750, 5.2250) node [] {-1}; 
\draw [gray!200] (9.6250, 5.2250) node [] {}; 
\draw [gray!200] (10.1750, 5.2250) node [] {}; 
\draw [gray!200] (10.7250, 5.2250) node [] {}; 
\draw [gray!200] (11.2750, 5.2250) node [] {-1}; 
\draw [gray!200] (11.8250, 5.2250) node [] {}; 
\draw [gray!200] (12.3750, 5.2250) node [] {}; 
\draw [gray!200] (12.9250, 5.2250) node [] {}; 
\draw [gray!200] (13.4750, 5.2250) node [] {}; 
\draw [gray!200] (7.9750, 4.6750) node [] {}; 
\draw [gray!200] (8.5250, 4.6750) node [] {-1}; 
\draw [gray!200] (9.0750, 4.6750) node [] {}; 
\draw [gray!200] (9.6250, 4.6750) node [] {}; 
\draw [gray!200] (10.1750, 4.6750) node [] {}; 
\draw [gray!200] (10.7250, 4.6750) node [] {-1}; 
\draw [gray!200] (11.2750, 4.6750) node [] {}; 
\draw [gray!200] (11.8250, 4.6750) node [] {}; 
\draw [gray!200] (12.3750, 4.6750) node [] {}; 
\draw [gray!200] (12.9250, 4.6750) node [] {}; 
\draw [gray!200] (13.4750, 4.6750) node [] {}; 
\draw [gray!200] (7.9750, 4.1250) node [] {-1}; 
\draw [gray!200] (8.5250, 4.1250) node [] {}; 
\draw [gray!200] (9.0750, 4.1250) node [] {}; 
\draw [gray!200] (9.6250, 4.1250) node [] {}; 
\draw [gray!200] (10.1750, 4.1250) node [] {-1}; 
\draw [gray!200] (10.7250, 4.1250) node [] {}; 
\draw [gray!200] (11.2750, 4.1250) node [] {}; 
\draw [gray!200] (11.8250, 4.1250) node [] {}; 
\draw [gray!200] (12.3750, 4.1250) node [] {}; 
\draw [gray!200] (12.9250, 4.1250) node [] {}; 
\draw [gray!200] (13.4750, 4.1250) node [] {}; 
\draw [gray!200] (7.9750, 3.5750) node [] {}; 
\draw [gray!200] (8.5250, 3.5750) node [] {}; 
\draw [gray!200] (9.0750, 3.5750) node [] {}; 
\draw [gray!200] (9.6250, 3.5750) node [] {-1}; 
\draw [gray!200] (10.1750, 3.5750) node [] {}; 
\draw [gray!200] (10.7250, 3.5750) node [] {}; 
\draw [gray!200] (11.2750, 3.5750) node [] {}; 
\draw [gray!200] (11.8250, 3.5750) node [] {}; 
\draw [gray!200] (12.3750, 3.5750) node [] {}; 
\draw [gray!200] (12.9250, 3.5750) node [] {}; 
\draw [gray!200] (13.4750, 3.5750) node [] {}; 
\draw [gray!200] (7.9750, 3.0250) node [] {}; 
\draw [gray!200] (8.5250, 3.0250) node [] {}; 
\draw [gray!200] (9.0750, 3.0250) node [] {-1}; 
\draw [gray!200] (9.6250, 3.0250) node [] {}; 
\draw [gray!200] (10.1750, 3.0250) node [] {}; 
\draw [gray!200] (10.7250, 3.0250) node [] {}; 
\draw [gray!200] (11.2750, 3.0250) node [] {}; 
\draw [gray!200] (11.8250, 3.0250) node [] {}; 
\draw [gray!200] (12.3750, 3.0250) node [] {}; 
\draw [gray!200] (12.9250, 3.0250) node [] {}; 
\draw [gray!200] (13.4750, 3.0250) node [] {-1}; 
\draw [gray!200] (7.9750, 2.4750) node [] {}; 
\draw [gray!200] (8.5250, 2.4750) node [] {-1}; 
\draw [gray!200] (9.0750, 2.4750) node [] {}; 
\draw [gray!200] (9.6250, 2.4750) node [] {}; 
\draw [gray!200] (10.1750, 2.4750) node [] {}; 
\draw [gray!200] (10.7250, 2.4750) node [] {}; 
\draw [gray!200] (11.2750, 2.4750) node [] {}; 
\draw [gray!200] (11.8250, 2.4750) node [] {}; 
\draw [gray!200] (12.3750, 2.4750) node [] {}; 
\draw [gray!200] (12.9250, 2.4750) node [] {-1}; 
\draw [gray!200] (13.4750, 2.4750) node [] {}; 
\draw [gray!200] (7.9750, 1.9250) node [] {-1}; 
\draw [gray!200] (8.5250, 1.9250) node [] {}; 
\draw [gray!200] (9.0750, 1.9250) node [] {}; 
\draw [gray!200] (9.6250, 1.9250) node [] {}; 
\draw [gray!200] (10.1750, 1.9250) node [] {}; 
\draw [gray!200] (10.7250, 1.9250) node [] {}; 
\draw [gray!200] (11.2750, 1.9250) node [] {}; 
\draw [gray!200] (11.8250, 1.9250) node [] {}; 
\draw [gray!200] (12.3750, 1.9250) node [] {-1}; 
\draw [gray!200] (12.9250, 1.9250) node [] {}; 
\draw [gray!200] (13.4750, 1.9250) node [] {}; 
\draw [gray!200] (7.9750, 1.3750) node [] {}; 
\draw [gray!200] (8.5250, 1.3750) node [] {}; 
\draw [gray!200] (9.0750, 1.3750) node [] {}; 
\draw [gray!200] (9.6250, 1.3750) node [] {}; 
\draw [gray!200] (10.1750, 1.3750) node [] {}; 
\draw [gray!200] (10.7250, 1.3750) node [] {}; 
\draw [gray!200] (11.2750, 1.3750) node [] {}; 
\draw [gray!200] (11.8250, 1.3750) node [] {-1}; 
\draw [gray!200] (12.3750, 1.3750) node [] {}; 
\draw [gray!200] (12.9250, 1.3750) node [] {}; 
\draw [gray!200] (13.4750, 1.3750) node [] {}; 
\draw [gray!200] (7.9750, 0.8250) node [] {}; 
\draw [gray!200] (8.5250, 0.8250) node [] {}; 
\draw [gray!200] (9.0750, 0.8250) node [] {}; 
\draw [gray!200] (9.6250, 0.8250) node [] {}; 
\draw [gray!200] (10.1750, 0.8250) node [] {}; 
\draw [gray!200] (10.7250, 0.8250) node [] {}; 
\draw [gray!200] (11.2750, 0.8250) node [] {-1}; 
\draw [gray!200] (11.8250, 0.8250) node [] {}; 
\draw [gray!200] (12.3750, 0.8250) node [] {}; 
\draw [gray!200] (12.9250, 0.8250) node [] {}; 
\draw [gray!200] (13.4750, 0.8250) node [] {}; 
\draw [gray!200, line width=0.4pt](7.7000,0.5500)--(13.7500,0.5500)--(13.7500,6.6000)--(7.7000,6.6000)--(7.7000,0.5500);
\draw [ gray!200, line width=0.4pt](11.2750,0.8250) circle (0.2200) ;
\draw [ gray!200, line width=0.4pt](11.8250,1.3750) circle (0.2200) ;
\draw [ gray!200, line width=0.4pt](7.9750,1.9250) circle (0.2200) ;
\draw [ gray!200, line width=0.4pt](12.3750,6.3250) circle (0.2200) ;
\draw [ gray!200, line width=0.4pt](12.9250,2.4750) circle (0.2200) ;
\draw [ gray!200, line width=0.4pt](13.4750,3.0250) circle (0.2200) ;
\draw [gray!200] (1.6500, 7.1500) node [] {$\overbrace{\hbox to 89.1000 pt{\qquad}}^{m}$}; 
\draw [gray!200] (9.9000, 0.0000) node [] {$\underbrace{\hbox to 118.8000 pt{\qquad}}_{2^k}$}; 
\draw [gray!200] (3.0250, 0.0000) node [] {$\underbrace{\hbox to 163.3500 pt{\qquad}}_{n}$}; 
\draw [gray!200] (9.6250, 4.6750) node [] {$g_{m-1}$}; 
\draw [gray!200] (7.9750, 6.8750) node [] {0}; 
\draw [gray!200] (8.5250, 6.8750) node [] {1}; 
\draw [gray!200] (9.0750, 6.8750) node [] {2}; 
\draw [gray!200] (9.6250, 6.8750) node [] {3}; 
\draw [gray!200] (10.1750, 6.8750) node [] {4}; 
\draw [gray!200] (10.7250, 6.8750) node [] {5}; 
\draw [gray!200] (11.2750, 6.8750) node [] {6}; 
\draw [gray!200] (11.8250, 6.8750) node [] {7}; 
\draw [gray!200] (12.3750, 6.8750) node [] {8}; 
\draw [gray!200] (12.9250, 6.8750) node [] {9}; 
\draw [gray!200] (13.4750, 6.8750) node [] {10}; 
\draw [ gray!200, line width=0.4pt](7.7000,1.6500)--(13.7500,1.6500);
\draw [ gray!200, line width=0.4pt](7.7000,2.2000)--(13.7500,2.2000);
\draw [ gray!200, line width=0.4pt](7.7000,3.3000)--(13.7500,3.3000);
\draw [ gray!200, line width=0.4pt](7.7000,6.0500)--(13.7500,6.0500);
\draw [ gray!200, line width=0.4pt](8.2500,0.5500)--(8.2500,6.6000);
\draw [ gray!200, line width=0.4pt](11.0000,0.5500)--(11.0000,6.6000);
\draw [ gray!200, line width=0.4pt](12.1000,0.5500)--(12.1000,6.6000);
\draw [ gray!200, line width=0.4pt](12.6500,0.5500)--(12.6500,6.6000);
\end{tikzpicture} 
\end{center}
\caption{Hankel determinant for $n=11, k=3, m=6$}
\label{fig:det}
\end{figure}

\begin{equation*}
	h_n= (-1)^{(n-m)(n-m+1)/2} h_m + (-1)^{(n-m)(n-m-1)/2+1} g_{m-1}.
\end{equation*}
	Since $m=2^{k+1}-n+1$ or $n-m=2n-2^{k+1}-1$, the above identity
	implies~\eqref{eq:hn}.
	\end{proof}

\begin{lemma}\label{th:sign}
	For each $n\geq 1$ we have $h_n \neq 0$ and
	\begin{equation}\label{eq:sign}
	\sign(h_n)=g_{n-1}.
	\end{equation}
\end{lemma}
\begin{proof}
	We prove \eqref{eq:sign} by induction on $n$.
	First we check that \eqref{eq:sign} is true for $n=1,2$.
	Suppose that \eqref{eq:sign} is true for $1,2,\ldots, n-1$.
	By Theorem \ref{th:main:h} and Lemma~\ref{th:g:g},  we have
\begin{equation*}
	h_n
	= (-1)^{n} (h_m + g_{m-1})
	= (-1)^{n} g_{m-1} (|h_m| + 1)
	= g_{n-1} (|h_m|+1),
\end{equation*}
	where $2^k< n\leq 2^{k+1}$ and  $ m=2^{k+1} - n + 1$.
	So $h_n \neq 0$ and $\sign(h_n)=g_{n-1}$.
\end{proof}
Lemma \ref{th:sign} and Theorem \ref{th:main:h} imply the following corollary
about the absolute values of the Hankel determinants of $B(x)$.
\begin{corollary}\label{th:abs:h}
	We have $|h_0|=|h_1|=1$ and for all $ n \geq 2$
\begin{equation}\label{eq:abs:h}
	|h_n| = |h_{2^{k+1}-n+1}|+1,
\end{equation}
where $2^k< n\leq 2^{k+1}$. 
\end{corollary}
Now we are ready to prove Conjectures 6, 7 and 10 of P. Barry \cite{barry}.
\begin{theorem}[Barry's Conjecture 6]\label{th:Barry6}
	For each $n\geq 0$,
	we have
	\begin{equation}
		|H_{n+1}(1-xr(x))|  = s_{n},
\end{equation}
	where the sequence $(s_n)_{n\geq 0}$ is defined in \eqref{def:sn}.
\end{theorem}
\begin{proof}
	Equivalently, it suffices to prove $h_1=1=s_0$ and
	\begin{equation}\label{eq:barry:need}
	|h_{n+1}|-|h_{n}| = j_n=(-1)^{m},
\end{equation}
	where $n=2^s (2m+1)\geq 1$. We consider two cases.

\bigskip

\noindent {(i)}  The case $n=2^{k+1}$, i.e., $m=0$: by \eqref{eq:abs:h}, we have
$$
	|h_{n+1}|=  |h_{2^{k+2}-n}| +1 = |h_{n}|+1.
$$

\bigskip

\noindent {(ii)} The case $2^k<n<2^{k+1}$ or $2^k+1<n+1\leq 2^{k+1}$: by \eqref{eq:abs:h}, we have
\begin{align*}
	|h_n|  &= |h_{2^{k+1}-n+1}| +1 ;\cr
	|h_{n+1}|  &= |h_{2^{k+1}-n}| +1 .
\end{align*}
So 
\begin{equation*}
	|h_{n+1}| - |h_{n}| = -( |h_{n'+1}| - |h_{n'}|),
\end{equation*}
	where $n'=2^{k+1} -n <n$. Hence we can prove \eqref{eq:barry:need} by induction on $n$.
	Since  $k\geq s+1$, and
	$$
	n'=2^{k+1}-n=2^{k+1} -2^s (2m+1) = 2^s(2(2^{k-s} -m-1) +1),
	$$
By the induction hypothesis, we get
$$
	|h_{n'+1}| - |h_{n'}| = (-1)^{2^{k-s} -m-1} = (-1)^{m+1},
$$
so that
	\begin{equation*}
	|h_{n+1}| - |h_{n}| = -( |h_{n'+1}| - |h_{n'}|) = (-1)^m.
		\end{equation*}
\end{proof}
Barry's Conjectures 7 and 10 are consequence of the above Theorem.

\begin{corollary}[Barry's Conjecture 8]\label{th:barry11}
	The sequence
\begin{equation*}
	u_n= \frac{|\sign(h_{n+1})-\sign(h_n)|}{2} \qquad (n\geq 1)
\end{equation*}
	is the paperfolding sequence on $\{0,1\}$, i.e., $u_{2n}=u_n, u_{4n+1}=1, u_{4n+3}=0$.
\end{corollary}
\begin{proof}
	By Lemma \ref{th:sign},
\begin{equation*}
	u_n= \frac{|g_{n}-g_{n-1}|}{2} \qquad (n\geq 1).
\end{equation*}
	By relation \eqref{eq:g:prop4}, we have
\begin{align*}
	u_{2n} &= \frac{1}{2}{\left|(-1)^{n(n+1)/2}g_n - (-1)^{n(n+1)/2}g_{n-1}\right|} = u_n; \cr
	u_{4n+1} &= \frac{1}{2}{\left|(-1)^{(2n+1)(2n+2)/2}g_{2n} - (-1)^{2n(2n+1)/2}g_{2n}\right|} = |g_{2n}|=1; \cr
	u_{4n+3} &= \frac{1}{2} {\left|(-1)^{(2n+2)(2n+3)/2}g_{2n+1} - (-1)^{(2n+1)(2n+2)/2}g_{2n+1}\right|} = 0. 
\end{align*}
\end{proof}

Barry's Conjecture 9 is true by Corollary \ref{th:barry11} (Barry's Conjecture 8) and the fact 
that $|a+b|+|a-b|=2$ for $a,b\in\{-1,1\}$.

\medskip

We give an algorithmic description of the sequence $H(B(x))$.
For $y \in {\mathbb Z}$, we define
\begin{equation}\label{def:y+}
	y^+=
	\begin{cases}
		y+1, & \text{if\ } y>0\\
		y-1. & \text{if\ } y\leq 0
	\end{cases}
\end{equation}
Note that for $|y^+| = |y| + 1$ and that $\sign(y^+) = \sign(y)$ for $y \neq 0$.

\begin{theorem}\label{th:h:B1}
The Hankel determinants $h_n$ of the sequence $B(x)$ are characterized by 
$h_0 = h_1 = 1$,$h_2=-2$ and

\medskip

	{\rm (a)} \  $h_{8n} = h_{4n}$        

	{\rm (b)} \  $h_{8n+1} = h_{4n+1}$

	{\rm (c)} \ $h_{8n+2} = h_{4n+2}$

	{\rm (d)} \ $h_{8n+3} = -(h_{4n+2})^+$

	{\rm (e)} \ $h_{8n+4} = -h_{4n+2}$

	{\rm (f)} \ $h_{8n+5} = -h_{4n+3}$

	{\rm (g)} \ $h_{8n+6} = (h_{4n+3})^+$

	{\rm (h)} \ $h_{8n+7} = h_{4n+3}$.

	\end{theorem}
\begin{proof}
The ``sign'' parts of (a)--(h) are easily derived from the four identities:
$$
\begin{aligned}
	\sign(h_{2n}) &= (-1)^{n(n+1)/2} \sign(h_{n}),  \\
	\sign(h_{2n+1}) &= (-1)^{n(n+1)/2} \sign(h_{n+1}) \\
	\sign(h_{4n+1}) &= (-1)^n \sign(h_{2n+1}) = - \sign(h_{4n+2}) \\
	\sign(h_{4n+4}) &= (-1)^{n+1} \sign(h_{2n+2}) = \sign(h_{4n+3}) \\
\end{aligned}
$$
The first two relations are consequences of \eqref{eq:sign} and \eqref{eq:g:prop4}, and they
immediately imply the other ones.

\bigskip

Now it suffices to prove the statement in Theorem~\ref{th:h:B1} for the ``absolute values'' parts.
We use induction on $n$, where the induction hypothesis ${\mathcal H}_n$ is: 
the relations (a) to (h) are true for all $(k, j)$ with $j \in [0,7]$ and $8k+j \leq 8n$.
It is easy to see that ${\mathcal H}_0$ is true. Now suppose ${\mathcal H}_n$ is true and let
us prove that the relations (b) to (h) hold and that $h_{8n+8} = h_{4n+4}$.
Let $n=2^s(2r+1)$. Using relations \eqref{eq:barry:need} and the induction hypothesis we have:

\medskip

	(b) Let $n=2^s(2m+1)$. Then we have
	$$|h_{8n+1}| = |h_{8n}| + (-1)^m = |h_{4n}| + (-1)^m = |h_{4n+1}|.$$
	
	(c) $|h_{8n+2}| = |h_{8n+1}| +1 = |h_{4n+1}| +1 = |h_{4n+2}|$.
	
	(d) $|h_{8n+3}| = |h_{8n+2}| +1 = |h_{4n+2}|+1 = |(h_{4n+2})^+|$.

	(e) $|h_{8n+4}| = |h_{8n+3}| -1= |(h_{4n+2})^+| - 1  = |h_{4n+2}|$.

	(f) $|h_{8n+5}| = |h_{8n+4}| +(-1)^n= |h_{4n+2}|+(-1)^n = |h_{4n+3}|$.

	(g) $|h_{8n+6}| = |h_{8n+5}| +1=  |h_{4n+3}|+1 = |(h_{4n+3})^+|$.

	(h) $|h_{8n+7}| = |h_{8n+6}| -1=  |(h_{4n+3})^+| - 1 = |h_{4n+3}|$.

\medskip

and finally

\medskip

	(a) $|h_{8n+8}| = |h_{8n+7}| -1=  |h_{4n+3}|-1 = |h_{4n+4}|$.
	
\end{proof}

The sequences $(g_n)_{n \geq 0}$ and $(h_n)_{n \geq 0}$ are somewhat related to the
(regular) paperfolding sequence, which is a $2$-automatic sequence. It is thus natural to ask 
whether they are automatic or regular (for more about $d$-automatic and $d$-regular sequences,
the reader can consult \cite{AS}, in particular Chapters 5 and 16). This question will be answered 
in the next two theorems.

\begin{theorem}\label{th:g:auto}
The sequence $(g_n)_{n \geq 0}$ is $2$-automatic.
\end{theorem}

\begin{proof}
To prove that the sequence $(g_n)_{n \geq 0}$ is $2$-automatic, we have to prove that the set
of subsequences $\{(g_{2^n+j}, \ n \geq 0, \ j \in [0,2^{k-1}]\}$ is finite. 
It suffices to prove the
following relations: for all $n \geq 0$
\begin{align*}
g_{4n+1} &= g_{2n+1} \ \ \ &g_{4n+2} &= g_{2n} \ \ \ &g_{4n+3} &= g_{2n} \\
g_{8n} &= g_{4n} \ \ \ &g_{16n+4} &= g_{2n+1} \ \ \ &g_{16n+12}& = g_{8n+4}. 
\end{align*}
We have seen in Equation~(\ref{eq:g:prop4}) that $(g_n)$ satisfies for all $n \geq 0$
$$
 g_{2n} = (-1)^{n(n+1)/2} g_n \ \ \text{and} \ \ g_{2n+1} = (-1)^{(n+1)(n+2)/2} g_n.
$$
Let us define $\gamma_n = (-1)^{n(n+1)/2}$ and $\delta_n = (-1)^{(n+1)(n+2)/2}$ so that
$g_{2n}= \gamma_n g_n$ and $g_{2n+1} = \delta_n g_n$. It easy to see that
$$
\begin{aligned}
\gamma_{2n} &= (-1)^n, \ \ \ \gamma_{2n+1} = \delta_{2n} = \delta_{2n+1} = (-1)^{n+1}, \\
\delta_{2n} \gamma_n& = \delta_n = -\gamma_{2n} \gamma_n, \ \ \ 
\gamma_{2n+1} \delta_n = \delta_{2n+1} \delta_n = \gamma_n.
\end{aligned}
$$
Thus
$$
\begin{aligned}
g_{4n+1} &= g_{2(2n)+1} = \delta_{2n} g_{2n} = \delta_{2n} \gamma_n g_{2n} 
= \delta_n \gamma_n^2 g_n = \delta_n g_n = g_{2n+1} \\
g_{4n+2} &= g_{2(2n+1)} = \gamma_{2n+1} g_{2n+1} =  \gamma_{2n+1} \delta_n g_n 
= \gamma_n g_n = g_{2n} \\
g_{4n+3} &= g_{2(2n+1)+1} = \delta_{2n+1} g_{2n+1} = \delta_{2n+1} \delta_n g_n 
= \gamma_n g_n = g_{2n} \\
g_{8n} &= g_{2(4n)} = \gamma_{4n} g_{4n} = (-1)^{2n} g_{4n} = g_{4n}.\\
\end{aligned}
$$
Now
$$
\begin{aligned}
g_{4n} &= g_{2(2n)} = \gamma_{2n} g_{2n}= \gamma_{2n} \gamma_n g_n 
= (-1)^n \gamma_n g_n \ \ \text{which implies} \\ 
g_{8n+4} &= g_{4(2n+1)} = - \gamma_{2n+1} g_{2n+1} = - \gamma_{2n+1} \delta_n g_n = 
- \gamma_n g_n.
\end{aligned}
$$
Hence
$$
\begin{aligned}
g_{16n+4} &= g_{8(2n)+4} = - \gamma_{2n} g_{2n} = - \gamma_{2n} \gamma_n g_n 
= \delta_n g_n = g_{2n+1} \\
g_{16n+12} &= g_{8(2n+1)+4} = - \gamma_{2n+1} g_{2n+1} = - \gamma_{2n+1} \delta_n g_n
= - \gamma_n g_n = g_{8n+4}.
\end{aligned}
$$

\end{proof}
 
Now we prove that the sequence $(h_n)_{n \geq 0}$ is $2$-regular (see \cite[Chapter 16]{AS} 
for more on $d$-regular sequences).

\begin{theorem}\label{th:h:regular}
The sequence $(h_n)_{n \geq 0}$ is $2$-regular.
\end{theorem}

\begin{proof}
To prove the $2$-regularity of $(h_n)_{n \geq 0}$, we establish,
using Theorem~\ref{th:h:B1},  the following equalities: 
for all $n \geq 0$.
$$
\begin{array}{lll}
&h_{8n} &= \ h_{4n} \\
&h_{8n+1} &= \ h_{4n+1} \\
&h_{8n+2} &= \ h_{4n+2} \\
&h_{8n+3} &= \ - h_{4n+1} - 2h_{4n+2} \\
&h_{8n+4} &= \ - h_{4n+2} \\
&h_{8n+5} &= \ - h_{4n+3} \\
&h_{8n+6} &= \ - h_{2n+1} + h_{4n+1} + h_{4n+2} + 2h_{4n+3} \\
&h_{8n+7} &= \ h_{4n+3}. \\
\end{array}
$$
All these equalities but two of them have been already proved in Theorem~\ref{th:h:B1}. 
It remains to prove: for all $n \geq 0$
$$
\begin{aligned}
{\rm (I)} \ h_{8n+3} &= - h_{4n+1} - 2h_{4n+2} \\
{\rm (II)} \ h_{8n+6} &= - h_{2n+1} + h_{4n+1} + h_{4n+2} + 2h_{4n+3} \\
\end{aligned}
$$
By \eqref{eq:barry:need} and Theorem \ref{th:h:B1}, we have
$$
	|h_{8n+3}| + |h_{4n+1}| = 2|h_{4n+2}|
	\text{\ and\ }
	\sign(h_{8n+3}) = -\sign(h_{4n+2}) = \sign(h_{4n+1})
$$
	This implies (I). Similary, we have
$$
	|h_{8n+6}| + |h_{4n+4}| = 2|h_{4n+3}|
	\text{\ and\ }
	\sign(h_{8n+6}) = \sign(h_{4n+3}) = \sign(h_{4n+4}).\\
$$
So 
\begin{equation}\label{eq:8n6:4n4}
h_{8n+6} = 2 h_{4n+3} - h_{4n+4}.
\end{equation}
On the other hand, we have
\begin{align*}\label{eq:8n6:remain}
	h_{4n+1} &= h_{8n+1} + h_{8n+2} + h_{8n+4};\\
	h_{4n+3} &= h_{8n+5} + h_{8n+6} + h_{8n+8}.
\end{align*}
Combining the above two identities yields
\begin{equation}\label{eq:8n6:remain}
	h_{2n+1} = h_{4n+1} + h_{4n+2} + h_{4n+4}.
\end{equation}
Identity~(II) is deduced from  \eqref{eq:8n6:4n4} and \eqref{eq:8n6:remain}.
\end{proof}

\begin{remark}
The $2$-automaticity of $(g_n)$ proved in Theorem~\ref{th:g:auto} can also be proved 
directly from the four identities at the beginning of the proof of Theorem~\ref{th:h:B1}
and the fact that the sequence $(-1)^n$, $(-1)^{n+1}$ and $(-1)^{n(n+1)/2}$ are periodic.

It is also possible to deduce the $2$-automaticity of $(g_{n-1})_{n \geq 1}$ (hence of
$(g_n)_{n \geq 0}$) from the relations satisfied by the $h_n$'s in Theorem~\ref{th:h:B1},
and Lemma~\ref{th:sign}. An important {\em caveat} is that it is not true in general
that the sequence of signs of a $d$-regular sequence is $d$-automatic; here is an example
concocted from \cite[p.~168--169]{reg1}: let $e_0(n)$ and $e_1(n)$ count,
respectively,
the number of $0$'s and the number of $1$'s in the binary expansion of the integer $n$.
It is easy to see that $(e_0(n))_n$ and $(e_1(n))_n$ are $2$-regular, so is $(f(n))_n$ defined
by $f(n) = e_0(n) - e_1(n)$. It is proved in \cite[p.~168--169]{reg1} that $(|f(n)|)_n$ is not
	$2$-regular. If $(\sign(f(n)))_n$ were $2$-automatic, hence $2$-regular, the product sequence
	$(\sign(f(n)) |f(n)|)_n$ would be $2$-regular as well. But $\sign(f(n)) |f(n)| = f(n)$, a contradiction.
\end{remark} 

\section{The Hankel determinants of $1+xr(x)$} 
\label{han2}

\begin{remark}
In this section we re-use the symbols $B(x), T(x), g_n, h_n$ 
with meanings different from those in the previous
section.  We trust there will be no confusion.
\end{remark}

Let
\begin{equation}\label{eqq:Bx}
	B(x)=1+xr(x)=1+\sum_{k\geq 0} x^{2^k} = 1+x+x^2+x^4+x^8+x^{16}+\cdots
\end{equation}
The first terms of the
Hankel determinants of $B(x)$ are 
$$
H(B(x))=(1, 1, 0, -1, 0, 1, 2, -1, 0, 1, 2, 3, -2, 1, 2, -1, 0, 1, 2, 3, \ldots)
$$
Consider the sequence $(r_1, r_2, r_3, \ldots)$ obtained from $B(x)$ by shifting two times, i.e.,
$$
T(x)=\frac{B(x)-(1+x)}{x^2}=\frac{r-1}{x}.
$$
It is also the shifted Rueppel sequence.  Let $g_n=H_n(T(x))$. 
From Lemma \ref{th:g:g} and relation \eqref{eq:g:prop4}, we have
$g_0=g_1=1$, and
\begin{equation}\label{eqq:g}
	g_{2n} = (-1)^{n(n-1)/2} g_n,
	\qquad
	g_{2n+1} = (-1)^{n(n+1)/2} g_n.
\end{equation}
and
\begin{equation}\label{eqq:g:g}	
g_n=(-1)^{n} g_{2^{k+1}-n-1},
\end{equation}	
where $2^k<n+1\leq 2^{k+1}$.

\medskip

Let $h_n=H_n(B(x))$.

\begin{theorem}\label{thh:main:h}
	We have $h_0=h_1=1$. For each $n\geq 2$, we have
\begin{equation}\label{eqq:hn}
	h_n= (-1)^{n-1} (h_m - g_{m-1}).
\end{equation}
	where $2^k< n\leq 2^{k+1}$ and  $ m=2^{k+1}-n+1$.
\end{theorem}
\begin{proof}
Our proof is	by the fundamental properties of determinants. 
Similar to the illustrated in Figure \ref{fig:det}
with the example $n=11, k=3, m=6$, we have
\begin{equation*}
	h_n= (-1)^{(n-m)(n-m-1)/2} h_m + (-1)^{(n-m-1)(n-m-2)/2+1} g_{m-1}.
\end{equation*}
	Since $m=2^{k+1}-n+1$, we have $n-m=2n-2^{k+1}-1$. The above identity
	implies~\eqref{eqq:hn}.
\end{proof}

\begin{lemma}\label{thh:sign}
	For each $n\geq 3$ we have 
	\begin{equation}\label{eqq:sign}
		\sign(h_n)=-g_{n-1},
\end{equation}
	with the convention that $\sign(0)=+1$.
\end{lemma}
\begin{proof}
	We prove \eqref{eqq:sign} by induction on $n$.
	First we check that \eqref{eqq:sign} is true for $n=3,4,5$.
	Suppose that \eqref{eqq:sign} is true for $3,4,\ldots, n-1$.
	Let $2^k< n\leq 2^{k+1}$ and  $ m=2^{k+1}-n+1$.
	We consider three cases.

\bigskip

\noindent {(i)} The  case $m=1$:
by	Theorem \ref{thh:main:h} and relation \eqref{eqq:g:g},  we have
\begin{equation*}
	h_n
	= (-1)^{n-1} (h_1 - g_{0})
	= 0.
\end{equation*}
	So $\sign(h_n)=1=-g_{n-1}$.

\bigskip

\noindent {(ii)} The  case $m=2$: we have
\begin{equation*}
	h_n
	= (-1)^{n-1} (h_2 - g_{1})
	= (-1)^n=-1.
\end{equation*}
	So $\sign(h_n)=-1=-g_{n-1}$.

\bigskip

\noindent {(iii)} The  case $m\geq 3$: we have
\begin{equation*}
	h_n
	= (-1)^{n-1} (h_m - g_{m-1})
	= (-1)^{n} g_{m-1} (|h_m| + 1)
	= -g_{n-1} (|h_m|+1),
\end{equation*}
	So $\sign(h_n)=-g_{n-1}$.
\end{proof}

Lemma \ref{thh:sign} and Theorem \ref{thh:main:h} imply the following corollary
about the absolute values of the Hankel determinants of $B(x)$.

\begin{corollary}\label{thh:abs:h}
	We have $|h_0|=|h_1|=1, |h_2|=0$, $|h_{2^{k+1}}|=0$, and
\begin{equation}\label{eqq:abs:h}
	|h_n| = |h_{2^{k+1}-n+1}|+1,
\end{equation}
	where $2^k< n< 2^{k+1}$.
\end{corollary}
Now we are ready to prove Conjecture 11 of P. Barry \cite{barry}.
\begin{theorem}[Barry's Conjecture 11]\label{thh:Barry11}
	For each $n\geq 1$ we have
	\begin{equation}
		|H_{n+1}(1-xr(x))|  = s_{n}-2,
\end{equation}
	where the sequence $(s_n)_{n\geq 0}$ is defined in \eqref{def:sn}.
\end{theorem}

\begin{proof}
	Equivalently, it suffices to prove $h_2=0=s_1-2$ and
	\begin{equation}\label{eqq:barry:need}
	|h_{n+1}|-|h_{n}| = j_n= (-1)^{m},
\end{equation}
	where $n=2^s (2m+1)\geq 2$. We consider
	three cases.

\bigskip

\noindent {(i)} The case $n=2^{k+1}$, i.e., $m=0$: by \eqref{eqq:abs:h}, we have
$$
	|h_{n+1}|=  |h_{2^{k+2}-n}| +1 = |h_{n}|+1.
$$

\bigskip

\noindent	{(ii)} The case $n=2^{k+1}-1$, i.e., $m=2^k-1$: by \eqref{eqq:abs:h}, we have
\begin{align*}
	|h_n|  &= |h_{2^{k+1}-n+1}| +1 = 1 ;\cr
	|h_{n+1}|  &= 0 .
\end{align*}
So 
$$
	|h_{n+1}|-|h_n|=-1=(-1)^m.
$$

\bigskip

\noindent {(iii)} The case $2^k<n<2^{k+1}-1$ or  $2^k+1<n+1< 2^{k+1}$: by \eqref{eqq:abs:h}, we have
\begin{align*}
	|h_n|  &= |h_{2^{k+1}-n+1}| +1 ;\cr
	|h_{n+1}|  &= |h_{2^{k+1}-n}| +1 .
\end{align*}
So 
\begin{equation*}
	|h_{n+1}| - |h_{n}| = -( |h_{n'+1}| - |h_{n'}|),
\end{equation*}
	where $n'=2^{k+1} -n <n$. Hence we can prove \eqref{eqq:barry:need} by induction on $n$.
	Since  $k\geq s+1$, and
	$$
	n'=2^{k+1}-n=2^{k+1} -2^s (2m+1) = 2^s(2(2^{k-s} -m-1) +1),
	$$
By the induction hypothesis, we get
$$
	|h_{n'+1}| - |h_{n'}| = (-1)^{2^{k-s} -m-1} = (-1)^{m+1},
$$
so that
	\begin{equation*}
	|h_{n+1}| - |h_{n}| = -( |h_{n'+1}| - |h_{n'}|) = (-1)^m.
	\end{equation*}
\end{proof}

We give an algorithmic description of the sequence $H(B(x))$. Recall that $y^+$
is defined by \eqref{def:y+} for each $y\in \mathbb{Z}$.

\begin{theorem}\label{th:h:B2}
The Hankel determinants $h_n$ of the sequence $B(x)$ are characterized by 
$h_0 = h_1 = 1$, $h_2=0$ and

\medskip

	{\rm (a)} \  $h_{8n} = h_{4n}$        

	{\rm (b)} \  $h_{8n+1} = h_{4n+1}$

	{\rm (c)} \ $h_{8n+2} = h_{4n+2}$

	{\rm (d)} \ $h_{8n+3} = (h_{4n+2})^+$

	{\rm (e)} \ $h_{8n+4} = -h_{4n+2}$

	{\rm (f)} \ $h_{8n+5} = -h_{4n+3}$

	{\rm (g)} \ $h_{8n+6} = -(h_{4n+3})^+$

	{\rm (h)} \ $h_{8n+7} = h_{4n+3}$.

	\end{theorem}
\begin{proof}
The ``sign'' parts of (a)--(h) are easily derived from the four identities:
$$
\begin{aligned}
	\sign(h_{2n}) &= (-1)^{n(n-1)/2} \sign(h_{n}),  \\
	\sign(h_{2n+1}) &= (-1)^{n(n-1)/2} \sign(h_{n+1}) \\
	\sign(h_{4n+1}) &= (-1)^n \sign(h_{2n+1}) =  \sign(h_{4n+2}) \\
	\sign(h_{4n+4}) &= (-1)^{n+1} \sign(h_{2n+2}) = -\sign(h_{4n+3}) \\
\end{aligned}
$$
(the first two relations are consequences of \eqref{eqq:sign} and \eqref{eqq:g}, and they
immediately imply the other ones).
Comparing Theorems \ref{th:Barry6} and \ref{thh:Barry11} we have 
$$
	|H_n(1+xr(x)) | = |H_n(1-xr(x))| -2,
$$
for $n\geq 2$.
The ``absolute values'' parts are treated as in the proof of 
	Theorem~\ref{th:h:B1}, using the above relation.  
\end{proof}

\begin{theorem}
The sequence $(g_n)_{n \geq 0}$ is $2$-automatic.
\end{theorem}

\begin{proof}
To prove that the sequence $(g_n)_{n \geq 0}$ is $2$-automatic, we have to prove that the set
of subsequences $\{(g_{2^n+j}, \ n \geq 0, \ j \in [0,2^{k-1}]\}$ is finite. It suffices to prove the
following relations: for all $n \geq 0$
\begin{align*}
g_{4n} &= g_{2n+1} \ \ \ &g_{4n+1} &= g_{2n+1} \ \ \ &g_{4n+2} &= g_{2n} \\
g_{8n+7} &= g_{4n+3} \ \ \ &g_{16n+3} &= g_{8n+3} \ \ \ &g_{16n+11}& = g_{2n}. 
\end{align*}
This can be done from Identity~(\ref{eqq:g}) by using the same method in Theorem \ref{th:g:auto}.
\end{proof}

\medskip

Now we prove that the sequence $(h_n)_{n \geq 0}$ is $2$-regular 
(see \cite[Chapter 16]{AS} for more on $d$-regular sequences).

\begin{theorem}
The sequence $(h_n)_{n \geq 0}$ is $2$-regular.
\end{theorem}

\begin{proof}
To prove the $2$-regularity of $(h_n)_{n \geq 0}$, we establish,
using Theorem~\ref{th:h:B1},  the following equalities: 
for all $n \geq 0$.
$$
\begin{array}{lll}
&h_{8n} &= \ h_{4n} \\
&h_{8n+1} &= \ h_{4n+1} \\
&h_{8n+2} &= \ h_{4n+2} \\
&h_{8n+3} &= \ - h_{4n+1} + 2h_{4n+2} \\
&h_{8n+4} &= \ - h_{4n+2} \\
&h_{8n+5} &= \ - h_{4n+3} \\
&h_{8n+6} &= \  h_{2n+1} - h_{4n+1} + h_{4n+2} - 2h_{4n+3} \\
&h_{8n+7} &= \ h_{4n+3}. \\
\end{array}
$$
This can be done by using the same method as in Theorem \ref{th:h:regular}.
\end{proof}

\section{The Hankel determinants of $r(x)/(r(x)-x)$} 
\label{han3}

Let us recall the following useful result \cite[Lemma 2.2]{Han2016Adv} 
\begin{lemma}\label{th:F:G}
Let $k$ be a nonnegative integer and let $F(x), G(x)$ be two power series 
satisfying
$$
	F(x)=\frac{x^k}{ 1+u(x)x - x^{k+2} G(x)},
$$
where $u(x)$ is a polynomial of degree less than or equal to $k$. Then
we have
$$
H_n(F)=(-1)^{k(k+1)/2} H_{n-k-1}(G). 
$$
\end{lemma}

Let
\begin{equation}\label{eqqq:Bx}
	B_0(x)=\frac{r(x)}{r(x)-x}= 
	1+x -x^4 + x^7 - x^8 -x^{10} +2x^{11} + \cdots
\end{equation}
The first terms of the
Hankel determinants of $B_0(x)$ are 
$$
H(B_0(x))=(1, 1, -1, 1, 1, -1, 1, 1, 1, -1, 1, -1, -1, -1, 1, 1, 1, \ldots)
$$

\begin{theorem}[Barry's Conjecture 16]\label{th:conj16}
	$$
	H_n\left(\frac{r(x)}{r(x)-x}\right)=\sign(H_n(1-xr(x))).
	$$
\end{theorem}

\begin{proof}
	We have 
	$$
	B_0(x)=\frac{1}{1-x-x^2 B_1(x)},
	$$
	where
	$$
	-B_1(x)=\frac{r-1}{xr} = \frac{1}{1+x -x^2 U(x)},
	$$
	and where
	$$
	U(x)=\frac{r-1-x}{(r-1)x^2}.
	$$
	Also,
	Let
	$$
G(x)=\frac{1-r(x)}{x}.
$$
Then
	$$
	-G(x)=\frac{1}{1-x^2 U(x)}.
	$$
	By Lemma \ref{th:F:G} we have, successively,
	\begin{align*}
		H_n(B_0)&=H_{n-1}(B_1),\\
		(-1)^{n}H_{n}(B_1)&=H_{n-1}(U),\\
		(-1)^{n}H_{n}(G)&=H_{n-1}(U).
\end{align*}
So by Lemma \ref{th:sign}, we derive  the equality
	\begin{equation*}	
	H_n(B_0)=\sign(H_n(1-xr)).
	\end{equation*}	
\end{proof}

\subsection*{Acknowledgments} The first author warmly thanks F. Durand for
discussions about Cobham's theorem and its 
generalizations.


\end{document}